\documentclass[11pt,a4paper,reqno]{amsart}

\usepackage{aliascnt,enumerate}
\usepackage{hyperref}
\usepackage{amsfonts}
\usepackage{mathrsfs}
\usepackage{amsmath}
\usepackage{amsmath,graphics}
\usepackage{amssymb}
\usepackage{indentfirst,latexsym,bm,amsmath,pstricks,amssymb,amsthm,graphicx,fancyhdr,float,color}
\usepackage[all]{xy}
\usepackage{amsthm}
\usepackage{amscd}
\usepackage[all]{hypcap}

\newtheorem{theorem}{Theorem}[section]

\newaliascnt{lemma}{theorem}
\newtheorem{lemma}[lemma]{Lemma}
\aliascntresetthe{lemma}

\newaliascnt{conjecture}{theorem}

\aliascntresetthe{conjecture}

\newaliascnt{proposition}{theorem}
\newtheorem{proposition}[proposition]{Proposition}
\aliascntresetthe{proposition}

\newaliascnt{corollary}{theorem}
\newtheorem{corollary}[corollary]{Corollary}
\aliascntresetthe{corollary}

\newaliascnt{problem}{theorem}

\aliascntresetthe{problem}

\newaliascnt{question}{theorem}
\newtheorem{question}[question]{Question}
\aliascntresetthe{question}

\newaliascnt{claim}{theorem}
\newtheorem{claim}[claim]{Claim}
\aliascntresetthe{claim}

\theoremstyle{definition}

\newaliascnt{definition}{theorem}
\newtheorem{definition}[definition]{Definition}
\aliascntresetthe{definition}

\newaliascnt{example}{theorem}
\newtheorem{example}[example]{Example}
\aliascntresetthe{example}

\newaliascnt{assumption}{theorem}
\newtheorem{assumption}[assumption]{Assumption}
\aliascntresetthe{assumption}

\theoremstyle{remark}

\newaliascnt{remark}{theorem}
\newtheorem{remark}[remark]{Remark}
\aliascntresetthe{remark}

\newaliascnt{remarks}{theorem}

\aliascntresetthe{remarks}

\numberwithin{equation}{section}

\numberwithin{figure}{section}

\def\wt{\widetilde}
\def\ol{\overline}

\def\lra{\longrightarrow}

\def\({$($}
\def\){$)$}

\def\bbp{\mathbb P}

\def\call{\mathcal L}

\def\Pic{\text{{\rm Pic\,}}}
\def\rank{\text{{\rm rank\,}}}

\def\Alb{{\rm Alb}}

\def\Alb{\mathrm{Alb}}

\begin{document}

	\title{Albanese fibrations of surfaces with low slope}	
	
	\author{Songbo Ling}
	
	\address{School of Mathematics, Shandong University, Jinan 250100, People's Republic of China}
	
	\email{lingsongbo@sdu.edu.cn}
	
	\author{Xin L\"u}
	
	\address{School of Mathematical Sciences,  Key Laboratory of MEA(Ministry of Education) \& Shanghai Key Laboratory of PMMP,  East China Normal University, Shanghai 200241, China}
	
	\email{xlv@math.ecnu.edu.cn}
	
	\thanks{This work is supported by National Natural Science Foundation of China,  Shanghai Pilot Program for Basic Research, Fundamental Research Funds for central Universities, Science and Technology Commission of Shanghai Municipality (No. 22DZ2229014) and  Natural Science Foundation of Shandong Province (No.  ZR2023QA00).}
	
	
	
	
	
	

	\begin{abstract}
		Let $S$ be a minimal irregular surface of general type, whose Albanese map induces a fibration $f:\,S \to C$ of genus $g$.
		We prove a linear upper bound on the genus $g$ if $K_S^2\leq 4\chi(\mathcal{O}_S)$, namely
		$$g\leq \left\{\begin{aligned}
			&6, &\quad&\text{if~}\chi(\mathcal{O}_S)=1;\\
			&3 \chi(\mathcal{O}_S)+1, && \text{otherwise}.
		\end{aligned}\right.$$
		Examples are constructed showing that the above linear upper bound is sharp.
		We also give a characterization of the Albanese fibrations reaching the above upper bound when $\chi(\mathcal{O}_S)\geq 5$.
		On the other hand, we will construct a sequence of surfaces $S_n$ of general type with $K_{S_n}^2/\chi(\mathcal{O}_{S_n})>4$ and with an Albanese fibration $f_n$, such that the genus $g_n$ of a general fiber of $f_n$ increases quadratically with $\chi(\mathcal{O}_{S_n})$,
		and that $K_{S_n}^2/\chi(\mathcal{O}_{S_n})$ can be arbitrarily close to $4$.
		%
	\end{abstract}
	
	\maketitle	
	
	\section{Introduction}\label{sec-intro}
	We work over the complex number $\mathbb{C}$ throughout this paper.
	Let $S$ be a minimal irregular surface of general type, and $a:\,S \to \Alb(S)$ be its Albanese map.
	When the image $a(S)$ is a curve, the Albanese map induces a fibration, which we call the Albanese fibration of $S$:
	$$f:\,S \longrightarrow C.$$
	In fact, by the universal property of the Albanese map, $C \cong a(S)$, and under this isomorphism the above fibration $f$ is nothing but the Albanese map of $S$.
	We are interested in the upper bound on the genus $g$ of a general fiber of $f$.
	\begin{question}
		Can we give an  upper bound on the genus $g$ of the Albanese fibration of $S$?
	\end{question}
	According to \cite{Cat00},
	the genus $g$ of the Albanese fibration of $S$ is a differential invariant, and hence is also a deformation invariant.
	Fixing $\chi(\mathcal{O}_S)$ and $K_S^2$, there are only finitely many deformation equivalence classes of such surfaces, cf. \cite[\S\,VII]{bhpv}.
	Hence there must be an upper bound of $g$ depending on $\chi(\mathcal{O}_S)$ or $K_S^2$. Such an argument is totally theoretic and gives no information how the upper bound of $g$ depends on $\chi(\mathcal{O}_S)$ or $K_S^2$.
	
	We are interested in the explicit upper bound on the genus $g$.
	If $g(C)>1$, then
	by the semi-positivity of the Hodge bundle $f_*\mathcal{O}_S\big(K_{S/C}\big)$,
	\begin{equation}\label{eqn-1-1}
		g\leq \frac{\chi(\mathcal{O}_S)}{g(C)-1}+1\leq \chi({\mathcal{O}_S})+1.
	\end{equation}
	The equality in \eqref{eqn-1-1} holds  if and only if $S$ is a generalized hyperelliptic surface, cf. \cite{Cat00}.
	If $g(C)=1$ and $K_S^2< 4\chi(\mathcal{O}_S)$, then from the slope inequality \cite{CH88,Xiao87} it follows that $\frac{K_S^2}{\chi(\mathcal{O}_S)}\geq \frac{4(g-1)}{g}$, and hence
	\begin{equation}\label{eqn-1-2}
		g\leq \frac{4\chi(\mathcal{O}_S)}{4\chi(\mathcal{O}_S)-K_S^2}\leq 4\chi(\mathcal{O}_S).
	\end{equation}
	Both of the above upper bounds are linear in $\chi(\mathcal{O}_S)$.
	It is natural to wonder how is the upper bound depending on $\chi(\mathcal{O}_S)$ when $K_S^2 \geq 4\chi(\mathcal{O}_S)$?
	Our first aim is to show that the slope $\frac{K_S^2}{\chi(\mathcal{O}_S)}\leq 4$ is a necessary and sufficient condition for the upper bound to be linear in $\chi(\mathcal{O}_S)$.
	\begin{theorem}\label{thm-main}
		Let $S$ be a minimal irregular surface of general type, and $f:\,S \to C$ be its Albanese fibration whose general fiber is of genus $g$.
		\begin{enumerate}[$(1)$]
			\item Suppose that $K_S^2 \leq 4\chi(\mathcal{O}_S)$. Then
			\begin{equation}\label{eqn-main}
				g\leq \left\{\begin{aligned}
					&6, &\quad&\text{if~}\chi(\mathcal{O}_S)=1;\\
					&3 \chi(\mathcal{O}_S)+1, && \text{otherwise}.
				\end{aligned}\right.
			\end{equation}
			Moreover, the above upper bound is sharp if $\chi(\mathcal{O}_S)\geq 2$. Namely, for any integer $\chi\geq 2$, there exists an Albanese fibration $f:\,S \to C$ whose general fiber is of genus $g$ such that\vspace{1mm}
			\begin{enumerate}[$($i$)$]
				\item ~$\chi(\mathcal{O}_S)=\chi$ and $K_{S}^2 \leq 4\chi(\mathcal{O}_{S})$;
				\item the genus $g=3 \chi(\mathcal{O}_{S})+1$. \vspace{3mm}
			\end{enumerate}
			\item For any $\lambda>4$, there exists a sequence of Albanese fibrations $f_n:\,S_n \to C_n$ such that\vspace{1mm}
			\begin{enumerate}[$(i)$]
				\item ~$\lim\limits_{n \to \infty}\chi(\mathcal{O}_{S_n})=+\infty$ \,and\, $\lim\limits_{n \to \infty}\frac{K_{S_n}^2}{\chi(\mathcal{O}_{S_n})}<\lambda $;\vspace{1mm}
				\item the genus $g_n$ increases quadratically with $\chi(\mathcal{O}_{S_n})$.
			\end{enumerate}
		\end{enumerate}
	\end{theorem}
	
	Of course, the slope $\frac{K_{S_n}^2}{\chi(\mathcal{O}_{S_n})}$ of the surfaces $S_n$ in (2) above is strictly larger than $4$ in view of \eqref{eqn-main}.
	It is natural to ask the geometry of the Albanese fibrations reaching the upper bound in \eqref{eqn-main}, or can we classify such Albanese fibrations?
	\begin{theorem}\label{thm-main-2}
		Let $S$ be a minimal irregular surface of general type with $K_S^2 \leq 4\chi(\mathcal{O}_S)$, and $f:\,S \to C$ be its Albanese fibration whose general fiber is of genus $g$.
		Suppose that $\chi(\mathcal{O}_S)\geq 5$ and that the equality holds in \eqref{eqn-main}. Then
		\begin{enumerate}[$(i)$]
			\item $K_S^2 = 4\chi(\mathcal{O}_S)$;
			\item the irregularity $q(S)=1$, and hence the base $C=E$ is an elliptic curve;
			\item the general fiber of $f$ is non-hyperelliptic;
			\item the canonical model $S_{can}$ is a flat double cover of a bielliptic surface $Y$ with the following commutative diagram:
			$$\xymatrix{ S \ar[drr]_-{f} \ar[rr] && S_{can} \ar[rr]^-{\pi} && Y \ar[dll]^-{h}\\
				&&E&&}$$
			\item the bielliptic surface $Y$ is the surface of type 6 appearing in Table\,2.1 in \autoref{sec-bielliptic};
			\item the branch divisor $R$ of $\pi$ admits at most negligible singularities.
		\end{enumerate}
	\end{theorem}
	\begin{theorem}\label{thm-main-3}
		For any fixed integer $\chi\geq 5$, let $\mathcal{M}=\mathcal{M}_{\chi}$ be the moduli space of minimal irregular surfaces $S$ of general type with $\chi(\mathcal{O}_S)=\chi$ reaching the equality in \eqref{eqn-main}.
		Then $\mathcal{M}$ is irreducible of dimension equal to $4\chi$.
	\end{theorem}
	
	\begin{remark}
		(1).~ Let $f:\, S \to C$ be a hyperelliptic Albanese fibration of genus $g$ with $K_S^2 \leq 4\chi(\mathcal{O}_S)$.
		According to the characterization in \autoref{thm-main-2}, there should be a better upper bound on $g$.
		In fact, in \cite[Theorem\,3.1]{ll24} it is proved that $g\leq \frac{4\chi(\mathcal{O}_S)+4}{2+\big(4\chi(\mathcal{O}_S)-K_S^2\big)} \leq 2\chi(\mathcal{O}_S)+2$;
		see also \cite[Theorem\,0.1]{Ish05} for the case when $K_S^2 = 4\chi(\mathcal{O}_S)$.
		
		(2). If the fibration $f:\,S \to C$ is not the Albanese fibration of $S$, then there is no upper bound of the genus $g$ depending on $\chi(\mathcal{O}_S)$ in general. Indeed, Penegini-Polizzi (\cite{PP17}) constructed a minimal surface $S$ with $p_g(S)=q(S)=2,K_S^2=4$ (whose Albanese map is generically finite), on which there are fibrations $f_k:\,S \to C$, such that the genera of $f_k$'s can be arbitrarily large.
		
		(3). Professor Z. Jiang kindly informed us that in a preprint \cite{jl24} joint with Lin they can also obtain some explicit upper bounds on the genus of the Albanese fibration when $p_g(S)=q(S)=1$ and $K_S^2=4$ or $5$.
	\end{remark}
	
	Let us briefly explain the main idea of the proofs.
	It consists of two main parts.
	When $g$ is large, it follows from the slope of non-hyperelliptic fibrations \cite{lu-zuo-18}
	that the Albanese fibration $f$ should be bielliptic, i.e., it is a double cover of another fibration $h:\,Y \to C$ whose general fiber is an elliptic curve.
	Together with the characterization of bielliptic fibrations with the minimal slope $4$ in \cite{barja01}, we can then obtain the bound \eqref{eqn-main} using the standard technique of the double covers between algebraic surfaces.
	
	On the other hand, if $g$ is small, it is not clear whether there is a double cover structure or not.
	However, as $g$ is small, we can assume $\chi(\mathcal{O}_S)$ is also small.
	In other words, such surfaces form finitely many bounded families.
	In particular, there are at most finitely many possibilities for the Harder-Narasimhan filtration of the Hodge bundle $f_*\omega_{S/C}$ associated to such an Albanese fibration $f:\,S \to C$.
	We then apply a case-by-case study with the help of Xiao's technique \cite{Xiao87} and the second multiplication map developed in \cite{lu-zuo-18} to finish the proof.
	
	To end the introduction, we would like to raise some related questions.
	\begin{question}
		What is the sharp upper bound on $g$ when $K_S^2\leq 4\chi(\mathcal{O}_S)$ and $\chi(\mathcal{O}_S)=1$?
	\end{question}
	By \eqref{eqn-main}, $g\leq 6$ in this case. We construct an example with $g=4$, see \autoref{exam-1}.
	However, we have no example with $g=5$ or $6$.
	On the other hand,
	It is proved that $g\leq 3$ if the $S$ is the standard isotrivial fibration (cf. \cite{cc91,penegini11,Pol09});
	it is also proved that $g\leq 4$ if either the Albanese fibration $f:\, S \to C$ is hyperelliptic (cf. \cite{Ish05,ll24}), or $\rank \mathcal{U}>0$ (cf. \autoref{coro-3-1}), where $\mathcal{U}$ is the unitary summand in the Fujita decomposition \eqref{eqn-fujita} of $f_*\omega_{S/C}$.
	
	\begin{question}
		Can we classify the Albanese fibrations with the equality in \eqref{eqn-main} when $\chi(\mathcal{O}_S)\leq 4$?
		We even don't know whether such Albanese fibrations are bielliptic or not.
	\end{question}
	
	\begin{question}
		How is the upper bound of $g$ depending on $\chi(\mathcal{O}_S)$ if $K_S^2 > 4\chi(\mathcal{O}_S)$?
	\end{question}
	The examples in \autoref{thm-main}\,(2) show that such an upper bound increases at least quadratically in $\chi(\mathcal{O}_S)$.
	In \cite{ll24} we have proved that the upper bound is indeed a quadratic function of $\chi(\mathcal{O}_S)$ if the Albanese fibration is hyperelliptic.
	Remark also that there is a linear upper bound if $g(C)=q(S)\geq 2$; see \eqref{eqn-1-1}.
	
	\vspace{2ex}
	The paper is organized as follows.
	In \autoref{sec-preliminaries}, we mainly review some basic facts about surface fibrations and do some technical preparations.
	In particular, the Xiao's technique and the second multiplication map are recalled.
	The upper bound \eqref{eqn-main} will be proved in \autoref{sec-proof-main}.
	In \autoref{sec-exam}, we construct examples showing that the linear upper bound in \eqref{eqn-main} is sharp, and also showing that the linear upper bound will be no longer true when $K_{S}^2>4\chi(\mathcal{O}_S)$.
	This completes the proof of \autoref{thm-main}.
	Finally in \autoref{sec-reaching} we study the Albanese fibrations reaching the upper bound \eqref{eqn-main}. In particular, \autoref{thm-main-2} and \autoref{thm-main-3} will be proved.
	

	\section{Preliminaries} \label{sec-preliminaries}
	In this section, we mainly review some basic facts and fix the notations.
	In \autoref{sec-fibration}, we recall some general facts about the surface fibrations, and refer to \cite{bhpv} for more details.
	In \autoref{sec-bielliptic} we restrict ourselves to the theory on the bielliptic fibrations; while in \autoref{sec-xiao-method}, we recall Xiao's technique and the second multiplication map, both of which are  crucial to \autoref{thm-main}.
	
	\subsection{The surface fibrations}\label{sec-fibration}
	Let $f:S\rightarrow C$ be a fibration of curves of genus $g\geq 2$, i.e. $f$ is a proper morphism from a smooth projective surface $S$ onto a smooth projective curve $C$ with connected fibers over the complex number, and the general fiber is a smooth projective curve of genus $g$. The fibration $f$ is called {\em relatively minimal} if there is no $(-1)$-curve (i.e. a smooth rational curve with self-intersection $-1$) contained in the fibers of $f$. It is called {\em hyperelliptic} if its general fiber   is hyperelliptic, {\em smooth} if all its fibers are smooth, {\em isotrivial} if all its smooth fibers are isomorphic to each other, and {\em locally trivial} if it is both smooth and isotrivial.
	
	Let $\omega_S$ (resp. $K_S$) be the canonical sheaf (resp. canonical divisor) of $S$, and let $\omega_f=\omega_{S/C}=\omega_S\otimes f^*\omega_C^\vee$ (resp. $K_f=K_{S/C}=K_S-f^*K_C$) be the relative canonical sheaf (resp. the relative canonical divisor) of $f$. Put $p_g(S):=h^0(S,\omega_S)$, $q(S):=h^1(S,\omega_S)$, $\chi(\mathcal{O}_S):=p_g(S)-q(S)+1$, and let $\chi_{top}(S)$ be the topological Euler characteristic of $S$. The basic invariants of
	$f$ are:
	$$\begin{aligned}
		\chi_f&\,=\chi(\mathcal{O}_S)-(g-1)(g(C)-1);\\
		K_f^2&\,=K_S^2-8(g-1)(g(C)-1);\\
		e_f&\,=\chi_{top}(S)-4(g-1)(g(C)-1).
	\end{aligned}$$
	These invariants satisfy the following properties:
	\begin{enumerate}
		\item $\chi_f=\deg f_*\omega_{S/C}$ is the degree of the Hodge bundle $f_*\omega_{S/C}$. Moreover $\chi_f\geq 0$, and the equality holds if and only if $f$ is locally trivial.
		\item When $f$ is relatively minimal, $K_f^2\geq 0$, and the equality holds if and only if $f$ is locally trivial.
		\item $e_f=\sum e_F$, where $e_F:=\chi_{\rm top}(F_{red})-(2-2g)$ for any fiber $F$ and $F_{red}$ is the reduced part of $F$. Moreover, $e_F\geq 0$, and the equality holds if and only if $F$ is smooth. Hence $e_f\geq 0$, and $e_f=0$ iff $f$ is smooth.
		\item The above three invariants satisfy the Noether equality: $12\chi_f=K_f^2+e_f$.
	\end{enumerate}
	
	Suppose that $f$ is relatively minimal and not locally trivial. Then both $K_f^2$ and $\chi_f$ are strictly positive. In this case, the {\em slope} of $f$ is defined to be $\lambda_f=\frac{K_f^2}{\chi_f}$.
	According to the non-negativity of these basic invariants of $f$,
	it holds $0<\lambda_f\leq 12$.
	The so-called slope inequality, proved independently by Cornalba-Harris \cite{CH88} and Xiao \cite{Xiao87}, states that
	\begin{equation}\label{eqn-slope}
		\lambda_f\geq \frac{4(g-1)}{g}.
	\end{equation}
	
	\subsection{Bielliptic surfaces}\label{sec-bielliptic}
	In this subsection, we recall some facts about bielliptic surfaces, and refer to \cite{bhpv,bea78,Ser90} for more details.
	Let $Y$ be a bielliptic surface.
	The Kodaira dimension of $Y$ is zero, and the invariants of $Y$ are as follows.
	$$K_Y\sim 0,\quad p_g(Y)=0,\quad q(Y)=1,\quad \chi(\mathcal{O}_S)=0,\quad h^{1,1}(Y)=2.$$
	Here ``$\sim$'' means ``numerically equivalent''.
	Moreover, there exist two elliptic curves $A$ and $B$, and an abelian group $G$ acting faithfully on both $A$ and $B$ such that:
	\begin{enumerate}
		\item the quotient $A/G\cong E$ is still an elliptic curve, and $B/G\cong \bbp^1$;
		\item the bielliptic surface $Y\cong (A \times B)/G$, where $G$ acts diagonally on the product $A \times B$.
	\end{enumerate}
	Let
	$$h=h_1:\,Y \lra E\cong A/G,\quad\text{and}\quad h_2:\,Y \lra \bbp^1\cong B/G,$$
	be the two natural projections.
	Then the fibers of $h_1$ are all smooth, and hence the elliptic fibration $h_1$ is locally trivial.
	However, the fibration $h_2$ admits singular fibers.
	Let $\Gamma_p$ be a singular fiber of $h_2$ over $p\in \bbp^1\cong B/G$.
	Then $p$ is a branch point of the finite map $B \to \bbp^1\cong B/G$,
	$\Gamma_p$ is a multiple of a smooth elliptic curve, and the multiplicity is equal to the ramification index of the finite map $B \to \bbp^1\cong B/G$ over $p$.
	In fact, all the smooth fibers of $h_1$ (resp. $h_2$) are isomorphic to $B$ (resp. $A$).
	By abuse of notations, denote also by $B, A$ the classes \big(in Num($Y$), $H^2(Y,\mathbb{Z})$, or $H^2(Y,\mathbb{Q})$\big) of the fibers of $h=h_1$ and $h_2$ respectively.
	Let $\gamma$ be the order of the group $G$. Then
	$$A^2=B^2=0,\qquad A\cdot B=\gamma.$$
	Clearly $\{A,B\}$ forms a basis of $H^2(Y,\mathbb{Q})$.
	However, the basis of Num($Y$) is a little subtle.
	It depends on the action of $G$, as exhibited by Serrano \cite[Theorem\,1.4]{Ser90} in the following table:\vspace{3mm}
	\begin{center}
		\begin{tabular}{|c|c|c|c|}
			\multicolumn{4}{c}{\bf Table 2.1}\\
			\multicolumn{4}{c}{}\\\hline
			\quad Type\quad\,  &\qquad\,$G$\qquad\,&\quad\, $\{m_1,\cdots,m_t\}$ \quad\, &\,\quad Basis of Num($Y$)\quad\,  \\\hline
			1 &$\mathbb{Z}_2$& $\{2,2,2,2\}$ & $\big\{(1/2)A,~B\big\}$ \\\hline
			2 &$\mathbb{Z}_2 \times \mathbb{Z}_2$& $\{2,2,2,2\}$ & $\big\{(1/2)A,~(1/2)B\big\}$ \\\hline
			3 &$\mathbb{Z}_4$& $\{2,4,4\}$ & $\big\{(1/4)A,~B\big\}$ \\\hline
			4 &$\mathbb{Z}_4\times \mathbb{Z}_2$& $\{2,4,4\}$ & $\big\{(1/4)A,~(1/2)B\big\}$ \\\hline
			5 &$\mathbb{Z}_3$& $\{3,3,3\}$ & $\big\{(1/3)A,~B\big\}$ \\\hline
			6 &$\mathbb{Z}_3\times \mathbb{Z}_3$& $\{3,3,3\}$ & $\big\{(1/3)A,~(1/3)B\big\}$ \\\hline
			7 &$\mathbb{Z}_6$& $\{2,3,6\}$ & $\big\{(1/6)A,~B\big\}$ \\\hline
		\end{tabular}
	\end{center}\vspace{3mm}
	Here $\mathbb{Z}_n$ stands for $\mathbb{Z}$ modulo $n\mathbb{Z}$,
	and the values $\{m_1,\cdots,m_t\}$ are the multiplicities of the singular fibers of $h_2:\, Y \to \bbp^1 \cong B/G$,
	which can be computed by applying the Riemann-Hurwitz formula to the map $B \to \bbp^1 \cong B/G$.

	\subsection{Xiao's technique and the second multiplication map}\label{sec-xiao-method}
	
	In this subsection, we briefly review
	Xiao's technique \cite{Xiao87} and the second multiplication map developed in \cite{lu-zuo-18}.
	Both techniques are based on the Harder-Narasimhan (H-N) filtration on the direct image sheaf $f_*\omega_f$.
	
	\begin{definition}\label{def-hn-filtra}
		For any locally free sheaf $\mathcal E$ on a smooth projective curve $B$,
		the slope of $\mathcal E$ is defined to be the rational number
		$\mu(\mathcal E)=\deg(\mathcal E)/{\rank(\mathcal E)}.$
		The sheaf $\mathcal E$ is said to be semi-stable,
		if for any coherent subsheaf $0\neq\mathcal E'\subsetneq\mathcal E$
		we have $\mu(\mathcal E')\leq\mu(\mathcal E)$.
		The Harder-Narasimhan (H-N) filtration of $\mathcal E$ is the following unique filtration:
		\begin{equation}\label{eqnharder-nara}
			0=\mathcal E_0\subset \mathcal E_1 \subset \cdots \subset \mathcal E_n=\mathcal E,
		\end{equation}
		such that:
		\begin{list}{}
			{\setlength{\labelwidth}{0.6cm}
				\setlength{\leftmargin}{0.7cm}}
			\item[(i)] the quotient $\mathcal E_i/\mathcal E_{i-1}$ is a locally free semi-stable sheaf for each $i$;
			\item[(ii)] the slopes are strictly decreasing
			$\mu(\mathcal E_i/\mathcal E_{i-1})>\mu( \mathcal E_{j}/\mathcal E_{j-1})$ if $i>j$.
		\end{list}
		The H-N filtration always exists.
		
		$\mathcal E$ is said to be positive (resp. semi-positive), if for any quotient sheaf  $\mathcal E \twoheadrightarrow \mathcal Q \neq 0$, one has $\deg \mathcal Q >0$ (resp. $\deg \mathcal Q \geq0$). Define $\mu_f(\mathcal E):=\mu( \mathcal E/\mathcal E_{n-1})$. Then $\mathcal E$ is positive (resp. semi-positive) if and only if $\mu_f(\mathcal E)>0$ (resp. $\mu_f(\mathcal E)\geq0$).
	\end{definition}
	
	Let now $f:\,S \to B$ be a surface fibration of genus $g\geq 2$, which is locally non-trivial.
	Take $\mathcal E=f_*\omega_f$,
	and in this case we write
	$$ E_i:=\mathcal E_i/\mathcal E_{i-1}, \quad \mu_i=\mu(E_i)~ (1\leq i\leq n); \mu_{n+1}:=0, \quad
	r_i=\rank(\mathcal E_i).$$

	Since $f_*\omega_f$ is semi-positive, we see $\mu_f(f_*\omega_f)=\mu_n\geq 0$.
	
	\begin{definition}[\cite{Xiao87}]\label{defofN(F)}
		Let $\mathcal E'$ be any locally free subsheaf of $f_*\omega_f$.
		The fixed and moving parts of $\mathcal E'$, denoted by $Z(\mathcal E')$ and $M(\mathcal E')$ respectively, are defined as follows.
		Let $\call$ be a sufficiently ample line bundle on $B$ such that the sheaf $\mathcal E'\otimes\call$ is generated by its global sections,
		and $\Lambda(\mathcal E')\subseteq |\omega_f\otimes f^*\call|$ be the linear subsystem corresponding to sections in $H^0(B,\,\mathcal E'\otimes\call)$.
		Then we define $Z(\mathcal E')$ to be the fixed part of $\Lambda(\mathcal E')$, and $M(\mathcal E')=\omega_f-Z(\mathcal E')$.
		Note that the definitions do not depend on the choice of $\call$.
	\end{definition}
	
	For a general fiber $F$ of $f$, let
	\begin{equation}\label{eqn-def-iota_i}
		\iota_i:~F \lra \Gamma_i \subseteq \mathbb P^{r_i-1}
	\end{equation}
	be the map defined by the restricted linear subsystem $\Lambda(\mathcal E_i)\big|_{F}$ on $F$ if $r_i\neq 1$,
	where $\mathcal E_i\subseteq f_*\omega_f$ is any subsheaf in
	the H-N filtration of $f_*\omega_f$ in \eqref{eqnharder-nara}.
	Let $d_i=M(\mathcal E_i)\cdot F$ and  $d_{n+1}:=2g-2$.
	
	\begin{definition}\label{def-double-cover}
		Let now $f:\,S \to B$ be a surface fibration of genus $g\geq 2$.
		Then $f$ is called a double cover fibration of type $(g,\gamma)$, if there is a rational two-to-one map $\pi:\,S \dashrightarrow Y$, where $h:\,Y \to B$ is a surface fibration of genus $\gamma$ and the following diagram commutes:
		$$\xymatrix{S \ar@{-->}[rr]^-{\pi} \ar[dr]_-{f} && Y \ar[dl]^-{h}\\
			&B&}$$
	\end{definition}
	\begin{remark}
		The above definition of double cover fibration is a little different from the one in \cite{barja01,cornalba-stoppino-08,lu-zuo-18}.
		In fact, the double cover fibration defined above is called a globally double cover fibration in \cite{lu-zuo-18}.
		A double cover fibration of type $(g,0)$ is just the hyperelliptic fibration; while a double cover fibration of type $(g,1)$ is also called a {\it bielliptic fibration}, which must be bielliptic in sense of \cite{barja01}.
		Conversely, a bielliptic fibration of genus $g\geq 6$ in sense of \cite{barja01} is also bielliptic in our sense. However, there exist examples of bielliptic fibrations in sense of \cite{barja01} but not in our sense, cf. \cite[Example\,1.2]{barja01}.
	\end{remark}
	
	The next lemma, which is due to Xiao, is crucial to the study of the slope of fibrations.
	\begin{lemma}[\cite{Xiao87}]\label{lemma-xiao}
		For any sequence of indices $1\leq i_1 <\cdots <i_{k}\leq n$, one has
		\begin{equation}\label{eqn-xiao}
			\omega_f^2\geq \sum_{j=1}^{k}\big(d_{i_j}+d_{i_{j+1}}\big)\big(\mu_{i_j}-\mu_{i_{j+1}}\big),\quad\text{where $i_{k+1}=n+1$.}
		\end{equation}
		In particular, one has
		\begin{equation}\label{eqn-xiao-1}
			\omega_f^2\geq \sum_{i=1}^{n}\big(d_{i}+d_{i+1}\big)\big(\mu_{i}-\mu_{i+1}\big).
		\end{equation}
	\end{lemma}
	
	The next two lemmas are due to Lu-Zuo \cite{lu-zuo-18,lu-zuo-19}.
	
	\begin{lemma}[{\cite[Equation (3.3)]{lu-zuo-18}}]\label{key-lemma-lu-zuo}
		Let $l=\min \{i|\iota_i$ is   birational $\}$.
		Assume that the general fiber $F$ is non-hyperelliptic. Then we have
		\begin{equation}\label{eqn-second-mult}
			\omega_f^2\geq \sum_{i=1}^{l-1}(3r_i-2)(\mu_i-\mu_{i+1})+\sum_{i=l}^{n}(5r_i-6)(\mu_i-\mu_{i+1}).
		\end{equation}
	\end{lemma}
	
	\begin{lemma}\label{lemma-lu-zuo-1}
		Let $S$ be a minimal surface of general type with $q=1,K^2=4\chi$, and its Albanese map $f:=a_S: S\rightarrow B=Alb(S)$ is a non-hyperelliptic fibration of genus $g\geq 6$.
		Suppose that  $\deg \iota_i\geq 2$. Then either
		\begin{enumerate}
			\item $\deg \iota_i=2$,  $f$ is a bielliptic fibration; or
			\item $d_i\geq \min \big\{ 3(r_i-1), 2(r_i-1)+ \frac{g}{2}\big\}$.
		\end{enumerate}
	\end{lemma}
	\begin{proof}
		If $\deg \iota_i\geq 3$, then by \cite[Lemma\,2.3]{lu-zuo-18},  we have $d_i\geq 3(r_i-1)$, so we are in case (ii).
		Suppose next that $\deg \iota_i=2$.
		Then $f$ is a double cover fibration of type $(g,\gamma_i)$. If $g\geq 4\gamma_i+1$, by \cite[Theorem\,4.9]{lu-zuo-19},
		we have $\gamma_i=1$, i.e., $f$ is a bielliptic fibration.
		If $g\leq  4\gamma_i$, then  by \cite[Lemma\,2.3]{lu-zuo-18}, we get
		\[d_i\geq \min\{4(r_i-1), 2(r_i+\gamma_i-1)\}\geq \min\big\{4(r_i-1), 2(r_i-1)+ \frac{g}{2}\big\}.\qedhere\]
	\end{proof}

	\section{The linear bound on the genus of the Albanese fibration}\label{sec-proof-main}
	In this section, we are aimed to prove the linear bound on the genus of the Albanese fibration with $K_S^2\leq 4\chi(\mathcal{O}_S)$;	
	namely to prove the upper bound \eqref{eqn-main} in \autoref{thm-main}.
	\begin{theorem}\label{thm-main(1)}
		Let $S$ be a minimal irregular surface of general type with $K_S^2 \leq 4\chi(\mathcal{O}_S)$, and $f:\,S \to C$ its Albanese fibration whose general fiber is of genus $g$. Then
		\begin{equation*}
			g\leq \left\{\begin{aligned}
				&6, &\quad&\text{if~}\chi(\mathcal{O}_S)=1;\\
				&3 \chi(\mathcal{O}_S)+1, && \text{otherwise}.
			\end{aligned}\right.
		\end{equation*}
	\end{theorem}
	After some elementary reductions, we will prove the above theorem based on \autoref{prop-3-1}, \autoref{prop-3-2}, and \autoref{prop-3-0}.
	The proofs of these propositions will be completed in \autoref{sec-proof-bielliptic}, \autoref{sec-proof-non-bielliptic} and
	\autoref{sec-first-red}.
	
	\begin{lemma}\label{lem-reduction}
		Let $S$ be a minimal irregular surface of general type with $K_S^2 \leq 4\chi(\mathcal{O}_S)$, and $f:\,S \to C$ its Albanese fibration whose general fiber is of genus $g$.
		\begin{enumerate}
			\item If $g(C)=q(S)\geq 2$, then $g\leq \chi(\mathcal{O}_S)+1$.
			\item If $f$ is hyperelliptic, then $g\leq 2\chi(\mathcal{O}_S)+2$.
		\end{enumerate}
	\end{lemma}
	\begin{proof}
		The first statement is already exhibited in \eqref{eqn-1-1};
		and the second one follows from
		\cite[Theorem\,3.1]{ll24} (see also \cite[Theorem\,0.1]{Ish05} for the case when $K_S^2 = 4\chi(\mathcal{O}_S)$).
	\end{proof}
	
	In the following, we will always assume that the base $C=E$ is an elliptic curve and that the Albanese fibration $f:\,S \to E$ is non-hyperelliptic
	unless otherwise stated.
	In particular, $q(S)=1$.
	
	\begin{proposition}\label{prop-3-1}
		Let $S$ be a minimal irregular surface of general type with $K_S^2 = 4\chi(\mathcal{O}_S)$, and $f:\,S \to E$ its Albanese fibration whose general fiber is of genus $g$. Suppose that $f$ is a bielliptic fibration. Then
		\begin{equation}\label{eqn-3-1}
			g\leq 3\chi(\mathcal{O}_S)+1.
		\end{equation}
	\end{proposition}
	
	\begin{proposition}\label{prop-3-2}
		Let $S$ be a minimal irregular surface of general type with $K_S^2 = 4\chi(\mathcal{O}_S)$, and $f:\,S \to E$ its Albanese fibration whose general fiber is of genus $g$. Suppose that $f$ is not bielliptic. Then
		\begin{equation}\label{eqn-3-2}
			g\leq \left\{\begin{aligned}
				&6, &\quad&\text{if~}\chi(\mathcal{O}_S)=1;\\
				&3 \chi(\mathcal{O}_S)+1, && \text{otherwise}.
			\end{aligned}\right.
		\end{equation}
	\end{proposition}
	
	\begin{proposition}\label{prop-3-0}
		Let $S$ be a minimal irregular surface of general type with $K_S^2 < 4\chi(\mathcal{O}_S)$, and $f:\,S \to E$ its Albanese fibration whose general fiber is of genus $g$. Then
		\begin{equation}\label{eqn-3-0}
			g\leq 3\chi(\mathcal{O}_S)+1.
		\end{equation}
	\end{proposition}
	
	\begin{proof}[{Proof of \autoref{thm-main(1)}}]
		It follows immediately from the above three propositions.
	\end{proof}

	\subsection{The bielliptic Albanese fibration}\label{sec-proof-bielliptic}
	In this subsection, we will first prove \autoref{prop-3-1}.
	Then we will confirm in \autoref{lemma-rkU-1} that
	the Albanese fibration $f$ would be bielliptic if either $g\geq 16$ or $f_*\omega_{S/E}$ has a non-trivial unitary summand with $g\geq 3$.
	
	\begin{proof}[{Proof of \autoref{prop-3-1}}]
		%
		If $g\leq 4$, then it is clearly true since $\chi(\mathcal{O}_S)\geq 1$. Now we assume $g\geq 5$.	Since the base $E$ is an elliptic curve,
		it follows that $K_f^2=K_S^2 = 4\chi(\mathcal{O}_S)=4\chi_f$.
		According to \cite[Theorem\,2.1]{barja01}, there is a double cover $\pi:\,S \to Y$ with the following commutative diagram
		$$\xymatrix{S \ar[rr]^-{\pi} \ar[d]_-{f} && Y \ar[d]^-{h}\\
			E \ar[rr]^-{=} && E}$$
		Moreover, the fibration $h:\, Y \to E$ is a locally trivial elliptic fibration,
		and the branch divisor $R\subseteq Y$ of $\pi$ admits at most negligible singularities.
		
		The local-triviality of $h$ implies that
		$$\chi(\mathcal{O}_Y)=0.$$
		Note also that $f$ is the Albanese fibration of $S$.
		Hence $q(S)=g(E)=1$,
		from which it follows that $q(Y)=g(E)=1$.
		Thus $p_g(Y)=\chi(\mathcal{O}_Y)+q(Y)-1=0$.
		Therefore $Y$ is a bielliptic surface.
		
		We use the notations introduced in \autoref{sec-bielliptic}. So
		there exist two elliptic curves $A$ and $B$,
		such that $Y\cong (A \times B) /G$, where $G$ acts diagonally on $A \times B$ with
		$A/G$ being an elliptic curve and $B/G\cong \bbp^1$.
		The projections of $A \times B$ to its two factors induce two elliptic fibrations on $Y$:
		$$h_1:\,Y \lra A /G,\qquad h_2:\,Y \lra \bbp^1.$$
		Note that the fibration $h:\,Y \to E$ is locally trivial.
		It follows that $A/G\cong E$, and that the fibration $h$ is the same as $h_1$.
		By abuse of notations, denote also by $A, B$ the classes \big(in Num($Y$), $H^2(Y,\mathbb{Z})$, or $H^2(Y,\mathbb{Q})$\big) of the fibers of $h_2$ and $h$ respectively.
		Let
		$$R \sim a A+bB.$$
		As the branch divisor $R\subseteq Y$ of $\pi$ admits at most negligible singularities,
		by the Riemann-Hurwitz formula, one obtains that (where $\gamma=A\cdot B$ is equal to the order of the group $G$)
		\begin{equation*}
			2g-2=R\cdot B=a\gamma, \quad \Longrightarrow \quad g=\frac{a\gamma}{2}+1;\qquad
		\end{equation*}
		
		\begin{equation}\label{eqn-3-10}	\chi(\mathcal{O}_S)=2\chi(\mathcal{O}_Y)+\frac12\Big(\frac12R+K_Y\Big)\cdot\frac12R=\frac{ab\gamma}{4}.
		\end{equation}
		Thus
		\begin{equation}\label{eqn-3-11}
			g=\frac{2\chi(\mathcal{O}_S)}{b}+1.
		\end{equation}
		On the other hand, the basis of Num($Y$) is given by
		$$\left\{\frac{A}{\alpha},~\frac{B}{\beta}\right\},$$
		where the integers numbers $\{\alpha,\beta\}$ are given in Table 1 in \autoref{sec-bielliptic},
		depending on the action of the group $G$.
		It follows that
		$$R \sim a A+bB=a\alpha\,\frac{A}{\alpha}+b\beta\,\frac{B}{\beta}.$$
		As $R$ is the branch divisor of the double cover $\pi:\,S \to Y$, $R$ is two-divisible. In particular,
		\begin{equation}\label{eqn-modui-need}
			2 ~\big|~b\beta, \quad \Longrightarrow\quad b\beta\geq 2.
		\end{equation}
		According to Table\,2.1 in \autoref{sec-bielliptic}, one sees that $\beta\leq 3$.
		Hence $b\geq \frac23$. Combining this with \eqref{eqn-3-11}, one proves the required bound \eqref{eqn-3-1}.
	\end{proof}
	
	\begin{remark}\label{rem-bielliptic}
		According to the proof above, if the equality in \eqref{eqn-3-1} holds,
		then $b=\frac{2}{3}$, and hence $\beta=3$.
		This implies that the bielliptic surface $Y$ is the surface of type 6 appearing in Table\,2.1 in \autoref{sec-bielliptic}.
	\end{remark}
	
	\begin{lemma}\label{lemma-rkU-1}
		Let $S$ be a minimal irregular surface of general type with $K_S^2 = 4\chi(\mathcal{O}_S)$, and $f:\,S \to E$ be its Albanese fibration whose general fiber is non-hyperelliptic of genus $g$.
		Then $f$ is a bielliptic fibration if any one of the following assumption holds.
		\begin{enumerate}
			\item The genus $g\geq 16$.
			\item The genus $g\geq 3$ and $\rank \mathcal{U} > 0$, where
			\begin{equation}\label{eqn-fujita}
				f_*\omega_{S/E}=\mathcal{A} \oplus \mathcal{U},
			\end{equation}
			is the Fujita decomposition of $f_*\omega_{S/E}$, where $\mathcal{A}$ is ample and $\mathcal{U}$ is unitary.
		\end{enumerate}
	\end{lemma}	
	\begin{proof}
		The first case follows from \cite{lu-zuo-18};
		indeed, since $f$ is assumed to non-hyperelliptic and $K_f^2=K_S^2 = 4\chi(\mathcal{O}_S)=4\chi_f$, it follows that $f$ must be bielliptic once $g\geq 16$ by \cite[Theorem\,1.5]{lu-zuo-18}.
		
		It remains to prove the statement under the second assumption.
		We first claim that $\rank \mathcal{U}=1$. Note that  $\mathcal{U}$ is a flat unitary vector bundle over an elliptic curve $E$.
		since $\pi_1(E)$  is abelian,  every unitary
		representation splits as a direct sum of 1-dimensional ones. Hence $\mathcal{U}$
		is a direct sum of torsion line bundles.
		Since $g(E)=q(S)=1$, $K_f^2=K_S^2 = 4\chi(\mathcal{O}_S)=4\chi_f$.
		It follows immediately that $\rank \mathcal{U} \leq 1$ by \cite{barja-zucconi-01}.
		
		Since $\rank \mathcal{U}=1$ and $E$ is an elliptic curve, it follows that $\mathcal{U}$ is a torsion line bundle,
		i.e., there exists a minimal integer $m$  such that $\mathcal{U}^{\otimes m}\cong \mathcal{O}_C$.
		Take an unramified  cyclic covering  $\phi:\, \tilde E\rightarrow E$  of order $m$ such that $\phi^*\mathcal{U}\cong \mathcal{O}_{\tilde E}$,
		and take   a base change of $f:\,S \to E$, then we get the following commutative diagram
			$$\xymatrix{\tilde{S} \ar[rr]^-{\tilde{\phi}} \ar[d]_-{\tilde{f}} && S \ar[d]^-{f}\\
				\tilde E \ar[rr]^-{\phi} && E}$$
			Here $\tilde{\phi}:\,\tilde{S}\rightarrow S$ is an unramified covering, so we have $K_{\tilde{S}}^2=m K_S^2$ and $\chi(\mathcal{O}_{\tilde{S}})=m\chi(\mathcal{O}_S)$. In particular
			we have $K_{\tilde{S}}^2=4\chi(\mathcal{O}_{\tilde{S}})$.
			On the other hand, note that  $\tilde{f}:\, \tilde{S}\rightarrow \tilde{E}$ is a fibration with general fiber  the same as $f:\, S\rightarrow E$.
			Put $\tilde{\mathcal{A}}:=\tilde{\phi}^*\mathcal{A}$,
			we have $\tilde{f}_*\omega_{\tilde{S}/\tilde{E}}=\tilde{f}_*\omega_{\tilde{S}}=\tilde{A}\oplus \mathcal{O}_{\tilde E}$, and thus  we get $q(S)=2$. Hence  $\tilde{S}$ is a minimal surface with maximal Albanese dimension and  $K_{\tilde{S}}^2=4\chi(\mathcal{O}_{\tilde{S}})$. By \cite{barja-pardini-stoppino-16} or \cite{lu-zuo-19-1}, we know the Albanese map
			$a_{\tilde{S}}:\,\tilde{S}\rightarrow A=Alb(\tilde{S})$ of $\tilde{S}$ is a double cover, whose branch divisor has at most negligible singularities.
			By the universal property of Albanese map, we see $\tilde{f}$ factors through $a_{\tilde{S}}$. In particular, $\tilde{f}$ is a double cover fibration
			$$
			\xymatrix{
				\wt S \ar[dr]^-{\tilde f}\ar[rr]^-{ a_{\wt S}}&&\wt A \ar[dl]_-{\tilde h}\\
				&\tilde E&
			}$$

			Let $\wt F$ be a general fiber of $\wt f$ and $\hat{E}:=a_{\wt S}(\wt F)$ be a general fiber of $\wt h$. Now we consider the H-N filtration of $\wt {\mathcal{E}}:=\wt f_*\omega_{\wt S/{\wt E}}$ in \autoref{sec-xiao-method}.
			Note $\wt {\mathcal{E}}_{n-1}=\wt {\mathcal{A}}$, $\tilde{r}_{n-1}:=\rank \wt {\mathcal{A}}=g-1$.
			\begin{claim}\label{claim-iota-n-1}
				The map $$\wt {\iota}_{n-1}: \wt F\rightarrow \wt \Gamma\subset  \mathbb{P}^{g-2}$$
				defined by $\Lambda(\wt {\mathcal{A}})|_{\wt F}$ in \autoref{sec-xiao-method}
				factors through $a_{\wt S}|_{\wt F}: \wt F\rightarrow \hat{E}$, where $a_{\wt S}|_{\wt F}$
				is the restriction of    $a_{\wt S}: \tilde{S}\rightarrow A$ to $\wt F$.
			\end{claim}
			\begin{proof}[{Proof of the \autoref{claim-iota-n-1}}]
				Since the canonical model of $\wt S$ is a flat double cover of the Abelian surface $A$,
				there exists an  ample line bundle $L$ on $A$ such that
				$$(a_{\wt S})_*\omega_{\wt S/{\wt E}}=\mathcal{O}_A\oplus \mathcal{O}_A(L),  \quad
				(a_{\wt S}|_{\wt F})_*\omega_{\wt F}=\mathcal{O}_{\hat{E}}\oplus \mathcal{O}_{\hat{E}}(L|_{\hat{E}}),
				\quad \wt h_*\mathcal{O}_A(L)=\mathcal{\wt {\mathcal{A}}}.$$

				Since  $\wt {\iota}_{n-1}$ is defined by $\Lambda(\wt {\mathcal{A}})|_{\wt F}$,  it factors through $\hat{E}$, i.e.
				\[\wt {\iota}_{n-1}: \wt F\xrightarrow{a_{\wt S}|_{\wt F}} \hat{E}\xrightarrow{|\mathcal{O}_{\hat {E}}(L|_{\hat {E}})|} \wt \Gamma \subset  \mathbb{P}^{g-2}\qedhere\]
			\end{proof}
			
			Now return back to the fibration $\tilde f:\, \tilde S \to \tilde E$ obtained by unramified base change as above.
			Note that the general fiber $\wt F$ of $\wt f$ is the same as a general fiber $F$ of $f$.
			Compare the H-N filtration of $\mathcal{E}=f_*\omega_{S/E}$ and $\wt {\mathcal{E}}:=\wt f_*\omega_{\wt S/{\wt E}}$, we see that $\iota_{n-1}: F\rightarrow \Gamma \subset  \mathbb{P}^{g-2}$ is the same as $\wt {\iota}_{n-1}: \wt F\rightarrow \wt \Gamma\subset  \mathbb{P}^{g-2}$. In particular,  the map
			$\varphi_{n-1}: S\rightarrow \mathbb{P}_E(\mathcal{A})$ defined by the linear system  $\Lambda(\mathcal{A})=\Lambda(\mathcal{E}_{n-1})$ factors through $Y$  and we have  the following commutative diagram
			$$
			\xymatrix{
				S \ar[dr]^-{f}\ar[rr]^-{\pi}&& Y \ar[dl]_-{h}\ar[r]^-{} &\mathbb{P}_E(\mathcal{A}) \ar[dll]^-{p}\\
				&E&
			}$$
			Therefore  $f: S\rightarrow E$ is a bielliptic fibration.
		\end{proof}

		\begin{remark}
			Let $S$ be a  minimal surface of general type with $q(S)=1, K_S^2=4\chi(\mathcal{O}_S)$ and a hyperelliptic Albanese fibration $f:\,S \to E$. If  $\rank \mathcal{U}>0$, then we have $g\leq 3$.
			In fact, since $g(E)=1$, we see $\lambda_f=4$.
			If  $\rank \mathcal{U}>0$, by the proof of \cite[Corollary 1.6]{lu-zuo-17},
			after a suitable unramified  base change, we get an irregular fibration $\wt f: \, \wt S\rightarrow E$ with $\lambda_{\wt f}=4$ and $q_{\wt f}=\rank \mathcal{U}$.
			Then we get $g\leq 3$ by \cite[Corollary \,1.5]{lu-zuo-17}.
		\end{remark}

		\begin{corollary}\label{coro-3-1}
			Let $S$ be a minimal irregular surface of general type with $K_S^2 = 4\chi(\mathcal{O}_S)$, and $f:\,S \to E$ be its Albanese fibration whose general fiber is of genus $g$.
			Suppose that either $g\geq 16$ or $\rank \mathcal{U} > 0$, where $\mathcal{U}$ is the unitary summand in the Fujita decomposition \eqref{eqn-fujita}. Then
			\begin{equation}\label{eqn-3-3}
				g\leq 3\chi(\mathcal{O}_S)+1.
			\end{equation}
		\end{corollary}
		\begin{proof}
			This follows directly from \autoref{prop-3-1} and \autoref{lemma-rkU-1}.
		\end{proof}

		\subsection{The Albanese fibration is not bielliptic}\label{sec-proof-non-bielliptic}
			In this subsection, we prove \autoref{prop-3-2}.
			It is based on Xiao's technique and the second multiplication map recalled in \autoref{sec-xiao-method}.
			Before going to the detailed proofs, we first do some reductions.
			According to \autoref{lem-reduction} and \autoref{coro-3-1}, to prove \autoref{prop-3-2}, we may assume in this subsection that
			\begin{assumption}\label{assumption}\mbox{}
				\begin{enumerate}
					\item The base curve $C=E$ is an elliptic curve, and hence $q(S)=1$;
					
					\item The Albanese fibration $f:\,S \to E$ is neither hyperelliptic nor bielliptic, and the general fiber of $f$ is of genus $g\leq 15$;
					
					\item The Hodge bundle $f_*\omega_{S/E}$ is ample, i.e., $\rank \mathcal{U}=0$, where $\mathcal{U}$ is the unitary summand in the Fujita decomposition as in \eqref{eqn-fujita};
					
					\item $K_S^2 = 4\chi(\mathcal{O}_S)$ and $\chi(\mathcal{O}_S)\leq 4$.
				\end{enumerate}
			\end{assumption}
			
			Since $g\leq 15$, the upper bound \eqref{eqn-3-2} follows immediately if $\chi(\mathcal{O}_S)\geq 5$.
			This is why we can make the assumption that $\chi(\mathcal{O}_S)\leq 4$ in the above.

			

			

			
			Recall the H-N filtration of  $\mathcal{E}=f_*\omega_{S/C}$  in \autoref{sec-xiao-method} .
			$$ E_i:=\mathcal E_i/\mathcal E_{i-1}, \quad \mu_i:=\mu(E_i)~ (1\leq i\leq n);~ \mu_{n+1}:=0, \quad
			r_i:=\rank(\mathcal E_i).$$
			For convenience, we set
			$$a_i=\rank E_i,\qquad b_i=\deg E_i.$$
			By definition, $\mu_i=\frac{b_i}{a_i}$.
			According to \autoref{assumption}, the Hodge bundle $f_*\omega_{S/E}$ is ample.
			In particular, we have $b_i=\deg E_i>0$ and   $\mu_i>0$  for $1\leq i\leq n$.
			It follows that
			\begin{equation}\label{eqn-4-1}
				\chi(\mathcal{O}_S)=\deg \mathcal{E}=\sum_{i=1}^{n} b_i\geq n.
			\end{equation}
			

			\begin{lemma}\label{lemma-n=1}
				Under \autoref{assumption},  if $n=1$, then  $g\leq 6$ .
			\end{lemma}
			\begin{proof}
				If $n=1$, we have $r_1=g, \mu_1=\frac{\chi(\mathcal{O}_S)}{g}$, then by \autoref{key-lemma-lu-zuo}, we have
				$$4\chi(\mathcal{O}_S)=\omega_f^2\geq (5g-6)\cdot \frac{\chi(\mathcal{O}_S)}{g}.$$
				Hence we get $4\geq 5-\frac{6}{g}$, thus $g\leq 6$.
			\end{proof}
			
			According to \eqref{eqn-4-1} together with \autoref{lemma-n=1},
			in order to prove \autoref{prop-3-2}, it remains to consider the cases when $\chi(\mathcal{O}_S)=4$, $3$, or $2$.
			Note also that $\omega_f^2=K_S^2=4\chi(\mathcal{O}_S)$ by assumption.
			The proofs of \autoref{prop-3-2} in the above three cases will be completed in
			\autoref{sec-chi-4},\autoref{sec-chi-3} and \autoref{sec-chi-2} respectively.
			
			%
			%
			%
			%
			%
				%
				%
				
				\subsubsection{The case when $\chi(\mathcal{O}_S)=4$ and $\omega_S^2=16$}\label{sec-chi-4}
				In this case, $n\leq 4$ by \eqref{eqn-4-1}.
				\autoref{prop-3-2} follows from \autoref{lemma-n=1} if $n=1$.
				It remains to consider the subcases if $n=2$, $n=3$ or $n=4$.
				
				\vspace{1ex}
				{\noindent \bf Case I: $n=2$.}~
				In this case, $\mu_1=\frac{b_1}{a_1}>\mu_2=\frac{b_2}{a_2}$.
				Note also that $b_1\geq 1,b_2\geq 1$ and  $b_1+b_2=4$.
				Hence there are $3$ possible choices for $(b_1,b_2)$:
				
				(case I-1): $b_1=3,b_2=1$; (case I-2): $b_1=b_2=2$; (case I-3): $b_1=1,b_2=3$.
				
				We will prove \autoref{prop-3-2} by contradiction.
				Suppose \eqref{eqn-3-2} does not hold,
				i.e., $g\geq 3\chi(\mathcal{O}_S)+2=14$.
				Note also that $g\leq 15$ by \autoref{assumption}, and that $\omega_f^2=K_S^2$ since $g(E)=1$.
				Combining with \eqref{eqn-xiao}, we have
				\begin{equation}\label{eqn-4-1-1}
					\mu_1+\mu_2=\frac{b_1}{a_1}+\frac{b_2}{a_2}<\frac{\omega_f^2}{2g-2}=\frac{8}{g-1}\leq \frac{8}{13}.
				\end{equation}
				
				Let $\iota_1$ be the map defined in \eqref{eqn-def-iota_i}.
				If the map $\iota_1$ is birational, by \autoref{key-lemma-lu-zuo} we have
				\begin{equation*}
					16=\omega_f^2\geq (5r_1-6)(\mu_1-\mu_2)+(5r_2-6)\mu_2=(5r_1-6)\mu_1+5(r_2-r_1)\mu_2=20-\frac{6b_1}{a_1}.
				\end{equation*}
				It follows that $\frac{b_1}{a_1}\geq \frac{2}{3}$,
				which gives a contradiction to \eqref{eqn-4-1-1}.

				
				
				
				
				If $\iota_1$ is not birational, by \autoref{key-lemma-lu-zuo}, we have
				\begin{equation}\label{eqn-4-1-3}
					16=\omega_f^2\geq (3r_1-2)(\mu_1-\mu_2)+(5r_1-6)\mu_2
					=3b_1+5b_2-\frac{2b_1}{a_1}+\frac{2b_2(a_1-2)}{a_2}.
				\end{equation}

				If moreover $\frac{g}{2}\geq a_1-1$,  we have $d_1\geq 3(a_1-1)$ by \autoref{lemma-lu-zuo-1}.
				Then by \autoref{lemma-xiao}, we have
				\begin{equation}\label{eqn-4-1-2}
					\begin{aligned}
						16=\omega_f^2 &\,\geq (d_1+d_2)(\mu_1-\mu_2)+(d_2+d_3)\mu_2\\
						&\,\geq  (3a_1+2g-2)\frac{b_1}{a_1}+(2g-2-3(a_1-1))\frac{b_2}{a_2}\\
						&\,=5b_1+2b_2+\frac{b_1(2a_2-5)}{a_1}-\frac{b_2(a_1-1)}{a_2}.
					\end{aligned}
				\end{equation}

				One can then prove case-by-case that it is impossible by combining \eqref{eqn-4-1-2} and \eqref{eqn-4-1-3} together with \eqref{eqn-4-1-1}.
				As an illustration, we exclude (case I-1) for instance.
				As $b_1=3$, $b_2=1$ and $14 \leq g=a_1+a_2 \leq 15$. By \eqref{eqn-4-1-1},  we have  $6\leq a_1\leq 11$,   $4\leq a_2\leq 9$.
				By \eqref{eqn-4-1-3}, we get
				$$a_1(a_1-2)\leq a_2(a_1+3).$$
				Hence we get $5\leq a_1\leq 9$ and if $a_1=9$, then $a_2=6$, $g=15$. In particular, we have
				$d_1\geq 3(a_1-1)$ by \autoref{lemma-lu-zuo-1}.
				Now by \eqref{eqn-4-1-2}, we have
				$$\frac{a_1-1}{a_2}\geq 1+\frac{3(2a_2-5)}{a_1},$$
				which gives a contradiction if $5\leq a_1\leq 9$.
				
				%
				%
				
				\vspace{1ex}
				{\noindent \bf Case II: $n=3$.}~
				In this case,  $\mu_1=\frac{b_1}{a_1}>\mu_2=\frac{b_2}{a_2}>\mu_3=\frac{b_3}{a_3}$. Note also that  $b_i\geq 1$ and  $b_1+b_2+b_3=4$, hence there are  3 possible choices for $(b_1,b_2,b_3)$:
				
				(case II-1): $b_1=2,b_2=b_3=1$; (case II-2): $b_2=2,b_1=b_3=1$; case (II-3): $b_3=2,b_1=b_2=1$.
				
				We will prove \autoref{prop-3-2} by contradiction.
				Suppose \eqref{eqn-3-2} does not hold.
				We can assume that $14\leq g\leq 15$ by \autoref{assumption}. Note $\omega_f^2=K_S^2=16$ since $g(E)=1$.
				Combined  with \eqref{eqn-xiao}, we get
				
				\begin{equation}\label{eqn-4-1-4}
					\mu_1+\mu_3=\frac{b_1}{a_1}+\frac{b_3}{a_3}<\frac{8}{g-1}\leq \frac{8}{13}.
				\end{equation}
				Note  $14\leq a_1+a_2+a_3\leq 15$ and  $a_1 <a_2<\frac{a_3}{2}$, by \eqref{eqn-4-1-4} we can exclude (case II-3); while in (case II-1) we get  $5\leq a_1\leq 6$, $3\leq a_2\leq 4$, $5\leq a_3\leq 7$ and in (case II-2) we get $a_1=3,~ a_2=7$, $4\leq a_3\leq 5$.
				
				Now we study (case II-1) and (case II-2).
				Let $\iota_2$ be the map defined in \eqref{eqn-def-iota_i}. Note $r_1= a_1\geq 2$, thus $5r_1-6\geq 3r_1-2$.
				If the map $\iota_2$ is birational,  by \autoref{key-lemma-lu-zuo} we have
				\begin{equation*}
					16\geq (3r_1-2)(\mu_1-\mu_2)+(5r_2-6)(\mu_2-\mu_3)+(5r_3-6)\mu_3=3b_1+5b_2+5b_3-\frac{2b_1}{a_1}+\frac{2b_2(a_1-2)}{a_2}.
				\end{equation*}
				Combined with  $\frac{b_1}{a_1}>\frac{b_2}{a_2}>\frac{b_3}{a_3}$, we get a contradiction to \eqref{eqn-4-1-4}.
				
				If $\iota_2$ is not birational, by \autoref{key-lemma-lu-zuo}, we have
				\begin{equation}\label{eqn-4-1-5}
					16=\omega_f^2\geq (3r_1-2)(\mu_1-\mu_2)+(3r_2-2)(\mu_2-\mu_3)+(5r_3-6)\mu_3=14-\frac{2b_1}{a_1}+\frac{2(a_1+a_2)-4}{a_3}.
				\end{equation}
				By \eqref{eqn-4-1-5}, we can exclude (case II-2) and the subcase $a_1=6$ in  (case II-1).

				For the subcase $a_1=5$ in (case II-1),   we have either $a_2=3,~ 6\leq a_3\leq 7$ or $a_2=4, ~5\leq a_3\leq 6$. Note $\frac{g}{2}\geq 7= r_2-1\geq r_1-1$.
				By \autoref{lemma-lu-zuo-1}, we have $d_1\geq 3(a_1-1)$ and $d_2\geq 3(a_1+a_2-1)$.
				Then by \autoref{lemma-xiao}, we have
				$$\begin{aligned}
					16=\omega_f^2 &\,\geq (d_1+d_2)(\mu_1-\mu_2)+(d_2+d_3)(\mu_2-\mu_3)+(d_3+d_4)\mu_3\\
					&\,\geq  (6a_1+3a_2-6)\frac{b_1}{a_1}+(2g-3a_1+1)\frac{b_2}{a_2}+(2g-3a_1-3a_2+1)\frac{b_3}{a_3}\\
					&\,=6b_1+2b_2+2b_3+\frac{b_1(2a_2-6)}{a_1}+\frac{b_2(2a_3-a_1+1)}{a_2}-\frac{b_3(a_1+a_2-1)}{a_3}.
				\end{aligned}$$
				Hence we get
				\begin{equation}\label{eqn-4-1-6}
					\frac{a_2+4}{a_3}\geq \frac{4a_2-12}{5}+\frac{2a_3-4}{a_2},
				\end{equation}
				which gives a contradiction.

				{\noindent \bf Case III: $n=4$.}~
				In this case, $b_i=1$ for $1\leq i\leq 4$,
				$\mu_1=\frac{1}{a_1}>\mu_2=\frac{1}{a_2}>\mu_3=\frac{1}{a_3}>\mu_4=\frac{1}{a_4}$,
				thus    $a_1<a_2<a_3<a_4$.
				
				We will prove \autoref{prop-3-2} by contradiction.
				Suppose \eqref{eqn-3-2} does not hold. We can assume that $14\leq a_1+a_2+a_3+a_4\leq 15$ by \autoref{assumption}.      Then we get  $a_1=2, a_2=3, a_3=4$ and $5\leq a_4\leq 6$.
				By  \eqref{eqn-xiao},   we have
				$$\frac{8}{12}=\frac{1}{2}+\frac{1}{6}\leq \mu_1+\mu_4<\frac{8}{g-1}\leq \frac{8}{13},$$
				which gives a contradiction.

				\subsubsection{The case when $\chi(\mathcal{O}_S)=3$ and $\omega_S^2=12$}\label{sec-chi-3}
				In this case, $n\leq 3$ by \eqref{eqn-4-1}.
				\autoref{prop-3-2} follows from \autoref{lemma-n=1} if $n=1$.
				It remains to consider the subcases if $n=2$ or $n=3$.

				\vspace{1ex}
				{\noindent \bf Case I: $n=2$.}~
				In this case,
				$\mu_1=\frac{b_1}{a_1}>\mu_2=\frac{b_2}{a_2}$. Note $b_1\geq 1,b_2\geq 1$ and  $b_1+b_2=3$, so we have 2 possible choices for $(b_1,b_2)$:
				
				(case I-1): $b_1=2,b_2=1$; (case I-2): $b_1=1, b_2=2$.
				
				We will prove \autoref{prop-3-2} by contradiction.
				Suppose \eqref{eqn-3-2} does not hold,
				i.e., $g\geq 3\chi(\mathcal{O}_S)+2=11$.
				Note also that $g\leq 15$ by \autoref{assumption}, and that $\omega_f^2=K_S^2=12$ since $g(E)=1$.
				Combining with \eqref{eqn-xiao}, we have
				\begin{equation}\label{eqn-4-2-1}
					\mu_1+\mu_2=\frac{b_1}{a_1}+\frac{b_2}{a_2}<\frac{\omega_f^2}{2g-2}=\frac{6}{g-1}\leq \frac{3}{5}.
				\end{equation}
				
				Let $\iota_1$ be the map defined in \eqref{eqn-def-iota_i}.
				If the map $\iota_1$ is birational, by \autoref{key-lemma-lu-zuo} we have
				\begin{equation*}
					12=\omega_f^2\geq (5r_1-6)(\mu_1-\mu_2)+(5r_2-6)\mu_2=15-\frac{6b_1}{a_1}.
				\end{equation*}
				It follows that $\frac{b_1}{a_1}\geq \frac{1}{2}$. Note $11\leq a_1+a_2\leq 15$ and $\frac{b_1}{a_1}>\frac{b_2}{a_2}$, combined with \eqref{eqn-4-2-1}
				we get  a contradiction.
				
				If $\iota_1$ is not birational, by \autoref{key-lemma-lu-zuo}, we have
				\begin{equation}\label{eqn-4-2-3}
					12=\omega_f^2\geq (3r_1-2)(\mu_1-\mu_2)+(5r_1-6)\mu_2
					=3b_1+5b_2-\frac{2b_1}{a_1}+\frac{2b_2(a_1-2)}{a_2}.
				\end{equation}

				If moreover $\frac{g}{2}\geq a_1-1$,
				then we have $d_1\geq 3(r_1-2)=3(a_1-1)$  by \autoref{lemma-lu-zuo-1}.
				Then by \autoref{lemma-xiao}, we have
				\begin{equation}\label{eqn-4-2-2}
					\begin{aligned}
						12=\omega_f^2 &\,\geq (d_1+d_2)(\mu_1-\mu_2)+(d_2+d_3)\mu_2\\
						&\,\geq  \big(3a_1+2g-2\big)\frac{b_1}{a_1}+\big(2g-2-3(a_1-1)\big)\frac{b_2}{a_2}\\
						&\,=5b_1+2b_2+\frac{b_1(2a_2-5)}{a_1}-\frac{b_2(a_1-1)}{a_2}.
					\end{aligned}
				\end{equation}	
				
				One can then prove case-by-case that it is impossible by combining \eqref{eqn-4-2-1} and \eqref{eqn-4-2-2} together with \eqref{eqn-4-2-3}.
				As an illustration, we exclude (case I-1) for instance.
				As $b_1=2$, $b_2=1$, $\frac{a_1}{2}<a_2$ and $11 \leq g=a_1+a_2 \leq 15$. By \eqref{eqn-4-2-1},  we have  $5\leq a_1\leq 9$, $4\leq a_2\leq 10$.
				
				By \eqref{eqn-4-2-3}, we get
				$$2a_1(a_1-2)\leq a_2(a_1+4).$$
				Hence we get $5\leq a_1\leq 7$ and if $a_1=7$, then $7\leq a_2\leq 8$. In particular, we have
				$\frac{g}{2}\geq a_1-1$ since $g\geq 11$.
				Now by \eqref{eqn-4-2-2}, we have
				$$a_1(a_1-1)\geq 2a_2(2a_2-5),$$
				which gives a contradiction if $5\leq a_1\leq 7$.
				
				\vspace{1ex}
				{\noindent \bf Case II: $n=3$.}~
				In this case, $b_1=b_2=b_3=1$, $\mu_1=\frac{1}{a_1}>\mu_2=\frac{1}{a_2}>\mu_3=\frac{1}{a_3}$, thus  $a_1<a_2<a_3$.
				
				We will prove \autoref{prop-3-2} by contradiction.
				Suppose \eqref{eqn-3-2} does not hold.   We can assume that $11\leq a_1+a_2+a_3+a_4\leq 15$ by \autoref{assumption}. Hence we get  $a_1\leq 4$.
				By  \eqref{eqn-xiao},   we have
				$$\mu_1+\mu_3=\frac{1}{a_1}+\frac{1}{a_3}<\frac{3}{g-1}\leq \frac{3}{5}.$$
				Combined with  $11\leq a_1+a_2+a_3\leq 15$,  we  get $3\leq a_1\leq 4$.

				Let $\iota_2$ be the map defined in \eqref{eqn-def-iota_i}. Note $r_1= a_1\geq 2$, thus $5r_1-6\geq 3r_1-2$.
				If the map $\iota_2$ is birational,  by \autoref{key-lemma-lu-zuo} we have
				\begin{equation*}
					12\geq (3r_1-2)(\mu_1-\mu_2)+(5r_2-6)(\mu_2-\mu_3)+(5r_3-6)\mu_3=13-\frac{2}{a_1}+\frac{2a_1-4}{a_2}.
				\end{equation*}
				Hence we get $\frac{2}{a_1}\geq 1+\frac{2a_1-4}{a_2}\geq 1$, contradicting to $3\leq a_1\leq 4$.
				
				If $\iota_2$ is not birational, by \autoref{key-lemma-lu-zuo}, we have
				\begin{equation}\label{eqn-4-2-4}
					12\geq (3r_1-2)(\mu_1-\mu_2)+(3r_2-2)(\mu_2-\mu_3)+(5r_3-6)\mu_3=11-\frac{2}{a_1}+\frac{2(a_1+a_2)-4}{a_3}.
				\end{equation}
				Note $a_1<a_2<a_3$ and $11\leq a_1+a_2+a_3+a_4\leq 15$,
				by \eqref{eqn-4-2-4}, we get $a_1=3,a_2=4,a_3\geq 6$.
				Since $\iota_1,\iota_2$ are not birational,  by \autoref{lemma-lu-zuo-1}
				we have
				$$d_1\geq 3(r_1-1)=6, \quad d_2\geq \min\{3(r_2-1), 2(r_2-1)+\frac{g}{2}\}\geq \frac{35}{2}.$$
				Since $d_2$ is an integer, we get $d_2\geq 18$.  Then by
				\autoref{lemma-xiao}, we have
				\begin{equation*}
					12\geq (d_1+d_2)(\mu_1-\mu_2)+(d_2+d_3)(\mu_2-\mu_3)+(d_3+d_4)\mu_3=6+\frac{g}{2}+\frac{2g-20}{a_3}
				\end{equation*}
				Hence we get $\frac{g}{2}+\frac{2g-20}{a_3}\leq 6$, which is a contradiction because of $g=7+a_3\geq 13$ here.
				
				\subsubsection{The case when $\chi(\mathcal{O}_S)=2$ and $\omega_f^2=8$}\label{sec-chi-2}
				In this case, $n\leq 2$ by \eqref{eqn-4-1}.
				\autoref{prop-3-2} follows from \autoref{lemma-n=1} if $n=1$.
				It remains to consider the subcase if $n=2$.
				In this case, $b_1=b_2=1$ and $\mu_1=\frac{1}{a_1}>\mu_2=\frac{1}{a_2}$, thus   $a_1<a_2$.
				
				We will prove \autoref{prop-3-2} by contradiction.
				Suppose \eqref{eqn-3-2} does not hold. We can assume that  $8\leq a_1+a_2+a_3+a_4\leq 15$ by \autoref{assumption}. Hence we get  $a_1\leq 4$.
				By  \eqref{eqn-xiao},   we have
				\begin{equation}\label{eqn-4-3-1}
					\mu_1+\mu_2=\frac{1}{a_1}+\frac{1}{a_2}<\frac{4}{g-1}\leq \frac{4}{7}.
				\end{equation}
				Combined with   $a_1+a_2\leq 15$, we get $a_1\geq 3$.
				
				Let $\iota_1$ be the map defined in \eqref{eqn-def-iota_i}. Note $r_1= a_1\geq 2$, thus $5r_1-6\geq 3r_1-2$.
				If the map $\iota_1$ is birational,  by \autoref{key-lemma-lu-zuo} we have
				\begin{equation*}
					8\geq (5r_1-6)(\mu_1-\mu_2)+(5r_2-6)\mu_2 =10-\frac{6}{a_1}.
				\end{equation*}
				Hence we get $\frac{6}{a_1}\geq 2$ and thus $a_1\leq 3$.
				
				Now assume $a_1=3$. Applying Castelnuovo's bound (\cite{lu-zuo-19} Lemma 3.3),
				we get $d_1\geq 6$. Now by \autoref{lemma-xiao} we get
				\begin{equation*}
					8\geq (d_1+d_2)(\mu_1-\mu_2)+(d_2+d_3)\mu_2\geq 4+\frac{2a_2+4}{3}-\frac{2}{a_2}
				\end{equation*}
				Hence we get $a_2(a_2-4)\leq 3$, contradicting to $a_1+a_2\geq 8$.
				
				\vspace{1ex}
				If $\iota_1$ is not birational, by \autoref{key-lemma-lu-zuo}, we have
				\begin{equation*}
					8\geq (3r_1-2)(\mu_1-\mu_2)+(5r_2-6)\mu_2 =8-\frac{2}{a_1}+\frac{2(a_1-2)}{a_2}.
				\end{equation*}
				Hence we get
				\begin{equation}\label{eqn-4-3-2}
					a_2\geq a_1(a_1-2),
				\end{equation}
				combined with \eqref{eqn-4-3-1} and  $a_1+a_2\leq 15$, we get   $3\leq a_1\leq 4$.
				
				Since $\iota_1$ is not birational, by  \autoref{lemma-lu-zuo-1},
				we get $d_1\geq 3(a_1-1)$. Then  by \autoref{lemma-xiao} we get
				\begin{equation*}
					8\geq (d_1+d_2)(\mu_1-\mu_2)+(d_2+d_3)\mu_2\geq 7+\frac{2a_2-5}{a_1}-\frac{a_1-1}{a_2}.
				\end{equation*}
				Hence we get
				$$a_2(2a_2-a_1-5)\leq a_1(a_1-1).$$
				If $a_1=3$,  we get     $a_2\leq 4$, contradicting to $a_1+a_2\geq 8$;  if $a_1=4$,   we get $a_2\leq 5$, contradicting  to \eqref{eqn-4-3-2}.
				
				\subsection{The case when $K_S^2 < 4\chi(\mathcal{O}_S)$}\label{sec-first-red}
				In this subsection, we deal with the case when $K_S^2 < 4\chi(\mathcal{O}_S)$.
				As illustrated in \autoref{sec-intro}, one can easily obtain a linear bound on the genus
				of the Albanese fibration using the slope inequality, cf. \eqref{eqn-1-2}.
				Here we want to improve this upper bound a little.
				
				\begin{proof}[{Proof of \autoref{prop-3-0}}]
					By \cite{lu-zuo-18}, one has $g\leq 15$.
					Suppose the conclusion does not hold.
					We have
					$$g\geq 3\chi(\mathcal{O}_S) +2.$$
					On the other hand, since we assume that $f$ is non-hyperellipitc by \autoref{lem-reduction},
					it follows that $\lambda_f=\frac{K_S^2}{\chi(\mathcal{O}_S)}>\frac{4(g-1)}{g}$ by \cite{kon-93}, which implies that
					$$g< \frac{4\chi(\mathcal{O}_S)}{4\chi(\mathcal{O}_S)-K_S^2}\leq 4\chi(\mathcal{O}_S),
					\quad\Longrightarrow\quad g\leq 4\chi(\mathcal{O}_S)-1.$$
					We may also assume that
					$$K^2_S=4\chi(\mathcal{O}_S)-1;$$
					otherwise, $4\chi(\mathcal{O}_S)-K_S^2\geq 2$ and hence $g< \frac{4\chi(\mathcal{O}_S)}{4\chi(\mathcal{O}_S)-K_S^2}\leq 2\chi(\mathcal{O}_S)$.
					Hence there are only two possibilities:
					\begin{enumerate}[(i)]
						\item $\chi(\mathcal{O}_S)=3, K_S^2=11$ and $g=11$;
						\item $\chi(\mathcal{O}_S)=4, K_S^2=15$ and $14\leq g\leq 15$.
					\end{enumerate}
					Moreover, by \cite{barja-zucconi-01}, we know that $f_*\omega_{S/E}$ is ample.  Then we can exclude case (i) and case (ii)
					using  a case-by-case study
					similarly as in \autoref{sec-chi-3} and  \autoref{sec-chi-4} respectively.
					
					As an illustration, we exclude case (i) here and leave the exclusion of case (ii) to the readers.
					In this case, $n\leq 3$ by \eqref{eqn-4-1}.
					\autoref{prop-3-0} follows from \autoref{lemma-n=1} if $n=1$.
					It remains to consider the subcases if $n=2$ or $n=3$.
					
					\vspace{1ex}
					{\noindent \bf Case I: $n=2$.}~
					In this case,
					$\mu_1=\frac{b_1}{a_1}>\mu_2=\frac{b_2}{a_2}$. Note $b_1\geq 1,b_2\geq 1$ and  $b_1+b_2=3$, so we have 2 possible choices for $(b_1,b_2)$:
					(case I-1): $b_1=2,b_2=1$; (case I-2): $b_1=1, b_2=2$.
					
					We will prove \autoref{prop-3-0} by contradiction.
					Suppose $\chi(\mathcal{O}_S)=3, K_S^2=11, g=11$.
					Note  $\omega_f^2=K_S^2=11$ since $g(E)=1$.
					Combined with \eqref{eqn-xiao}, we get
					\begin{equation}\label{eqn-3-2-1}
						\mu_1+\mu_2=\frac{b_1}{a_1}+\frac{b_2}{a_2}<\frac{\omega_f^2}{2g-2}=\frac{11}{20}.
					\end{equation}
					
					Let $\iota_1$ be the map defined in \eqref{eqn-def-iota_i}.
					If the map $\iota_1$ is birational, by \autoref{key-lemma-lu-zuo} we have
					\begin{equation*}
						11=\omega_f^2\geq (5r_1-6)(\mu_1-\mu_2)+(5r_2-6)\mu_2=15-\frac{6b_1}{a_1}.
					\end{equation*}
					It follows that $\frac{b_1}{a_1}\geq \frac{2}{3}$, contradicting to \eqref{eqn-3-2-1}.

					If $\iota_1$ is not birational,  by \autoref{lemma-lu-zuo-1}, we have $d_1\geq 3(r_1-2)=3(a_1-1)$.
					Then by \autoref{lemma-xiao}, similar to \eqref{eqn-4-2-2}, we have
					
					\begin{equation}\label{eqn-3-2-2}
						11+\frac{b_2(a_1-1)}{a_2}\geq 5b_1+2b_2+\frac{b_1(2a_2-5)}{a_1}.
					\end{equation}

					One can then prove case-by-case that it is impossible by combining \eqref{eqn-3-2-1} and \eqref{eqn-3-2-2}.
					
					(case I-1):
					As $b_1=2$, $b_2=1$, $\frac{a_1}{2}<a_2$ and $a_1+a_2=g=11$. By \eqref{eqn-3-2-1},  we have  either $a_1=6, a_2=5$ or  $a_1=7, a_2=4$.
					By \eqref{eqn-3-2-2}, we have
					$$\frac{a_1-1}{a_2}\geq 1+\frac{2(2a_2-5)}{a_1},$$
					which gives a contradiction if $a_1=6, a_2=5$ or  $a_1=7, a_2=4$.
					
					(case I-2): As $b_1=1$, $b_2=2$, $\frac{a_1}{2}<a_2$ and $a_1+a_2=g=11$, we see $a_1\leq 3$.   Then \eqref{eqn-3-2-1} gives a contradiction if $a_1\leq 3$.

					\vspace{1ex}
					{\noindent \bf Case II: $n=3$.}~
					In this case, $b_1=b_2=b_3=1$, $\mu_1=\frac{1}{a_1}>\mu_2=\frac{1}{a_2}>\mu_3=\frac{1}{a_3}$, thus  $a_1<a_2<a_3$.
					Suppose $\chi(\mathcal{O}_S)=3, \omega_f^2=K_S^2=11, g=11$.     Then  we get  $a_1\leq 2$ and $a_3\leq 6$. Hence we have $\frac{1}{a_1}+\frac{1}{a_3}\geq \frac{1}{2}+\frac{1}{6}=\frac{2}{3}$.
					By  \eqref{eqn-xiao},   we have
					$$\mu_1+\mu_3=\frac{1}{a_1}+\frac{1}{a_3}<\frac{\omega_f^2}{2g-2}=\frac{11}{20},$$
					which gives a contradiction.
				\end{proof}
				
				\section{Examples}\label{sec-exam} 
				In this section, we construct several examples:
				in \autoref{exam-1} we construct a sequence of examples showing that the upper bound obtained in \autoref{thm-main} is sharp;
				in \autoref{exam-2} and \autoref{exam-4-2} we construct a sequence of examples where the genera of the Albanese fibrations increase quadratically with $\chi(\mathcal{O}_S)$ (and $K_S^2$) and the quotient $\frac{K_S^2}{\chi(\mathcal{O}_S)}$ can be arbitrarily close to $4$.
				These examples together with \autoref{thm-main(1)} complete the proof of \autoref{thm-main}.
				\begin{proof}[{Proof of \autoref{thm-main}}]
					(1). The upper bound \eqref{eqn-main} is proved in \autoref{thm-main(1)}.
					The examples constructed in \autoref{exam-1} with $\chi\geq 2$ give the required Albanese fibrations reaching the upper bound in \eqref{eqn-main}.
					
					(2). Let $x$ be any fixed even integer such that $\frac{8(x+1)}{2x+1}<\lambda$, and let $k=2n$ in \autoref{exam-2}.
					Then one obtains a sequence of Albanese fibrations $f_n:\,S_n \to E_n$ over an elliptic curve, such that the genus of a general fiber is $g_n=2n^2+2xn+\frac{x^2-4}{2}$ and
					$$K_{S_n}^2=8(x+1)n+2x^2-14,\qquad \chi(\mathcal{O}_{S_n}) =(2x+1)n+\frac{x^2-4}{2}.$$
					In other words, we obtain a sequence of surfaces $S_n$ of general type with an Albanese fibration of $g_n$ increasing quadratically with $\chi(\mathcal{O}_{S_n})$,
					i.e.,
					$$g_n=\frac{2}{(2x+1)^2}\chi(\mathcal{O}_{S_n})^2+\frac{2(x^2+x+4)}{(2x+1)^2}\chi(\mathcal{O}_{S_n})+\frac{(x-2)(x-1)(x+2)(x+3)}{2(2x+1)^2}.$$
					Moreover,
					$$\frac{K_{S_n}^2}{\chi(\mathcal{O}_{S_n})}=\frac{8(x+1)n+2x^2-10}{(2x+1)n+\frac{x^2-4}{2}} ~\longrightarrow~ \frac{8(x+1)}{2x+1}<\lambda,\qquad \text{when~}n \longrightarrow +\infty.$$
					This completes the proof.
				\end{proof}
				
				\begin{example}\label{exam-1}
					For any integer $\chi\geq 1$, we construct an Albanese fibration $f:\,S \to E$ over an elliptic curve, whose general fiber is non-hyperelliptic of genus $g$ such that\vspace{1mm}
					$$\chi(\mathcal{O}_S)=\chi,\qquad K_{S}^2 = 4\chi(\mathcal{O}_{S}),\quad \text{~and~} \quad g=3 \chi(\mathcal{O}_{S})+1.$$
				\end{example}
				Let $Y$ be the bielliptic surface of type 6 in Table 1 in \autoref{sec-bielliptic}.
				From the exponential sequence one obtains the long exact sequence
				$$0=H^0(Y,\mathcal{O}_Y) \to H^1(Y,\mathcal{O}_Y^*) \to H^2(Y,\mathbb{Z}) \to H^2(Y,\mathcal{O}_Y) =0.$$
				Hence there exist (integral) divisors mapped to $(1/3)A,\,(1/3)B$ in Num($Y$),
				which we denote by $A_0,\,B_0$ respectively.
				Consider the linear system $\big|2(\chi\,A_0+B_0)\big|$.
				Since $K_Y \sim 0$, by Reader's method \cite{reider88}, one proves easily that the linear system
				$$\big|2(\chi\,A_0+B_0)\big|=\big|K_Y+\big(2(\chi\,A_0+B_0)-K_Y\big)\big|$$
				is base-point-free.
				According to Bertini's theorem, one can find a smooth divisor
				$$R_{\chi}\in \big|2(\chi\,A_0+B_0)\big|.$$
				Thus one can construct a double cover $\pi:\,S \to Y$ branched exactly over the above smooth divisor $R_{\chi}$.
				Let
				$$f:\, S \lra E,$$
				be the fibration induced by the elliptic fibration $h=h_1:\, Y \to E \cong A/G$.
				Let $g$ be the genus of a general fiber of $f$.
				According to the Riemann-Hurwitz formula (In this case, $\gamma=|G|=9$),
				$$\begin{aligned}
					2g-2&\,=\,R\cdot B=2nA_0\cdot B=2\chi\cdot \frac{\gamma}{3}=6\chi,\quad \Longrightarrow \quad g=3\chi+1,\\
					K_{S}^2&\,=\,2\Big(K_Y+\frac12R_{\chi}\Big)^2=4\chi,\\
					\chi(\mathcal{O}_{S})&\,=\,2\chi(\mathcal{O}_Y)+\frac12\Big(\frac12R_{\chi}+K_Y\Big)\cdot\frac12R_{\chi}=\chi.
				\end{aligned}
				$$
				Hence $K_{S}^2=4\chi(\mathcal{O}_{S})$ and
				$g=3\chi(\mathcal{O}_{S})+1$ as required.
				It is clear that the branch divisor $R_\chi$ is ample.
				It follows that $q(S)=q(Y)=1=g(E)$,
				which implies that $f$ is nothing but the Albanese map of $S$;
				namely, $f$ is the Albanese fibration of $S$.
				Finally, As the general fiber of $f$ is a double cover of an elliptic curve with $g\geq 4$,
				one deduces that $f$ is non-hyperelliptic by \cite[Lemma\,7]{Xiao87}.
				
				\begin{example}\label{exam-2}
					For any integers $k\geq 2$ and $x\geq 2$ with $k+x$ being even, we construct an Albanese fibration $f:\,S \to E$ over an elliptic curve, such that
					\begin{equation}\label{eqn-exam-2}
						\left\{
						\begin{aligned}
							&~g=\frac{k^2}{2}+xk+\frac{x^2-4}{2},\\
							&~K_S^2=4(x+1)k+2x^2-14,\\
							&~\chi(\mathcal{O}_S) =\Big(x+\frac12\Big)k+\frac{x^2-4}{2}.
						\end{aligned}\right.
					\end{equation}
				\end{example}
				The construction is stimulated by \cite[Example\,5.1]{ll24}.
				For any integer $m\geq 1$, let  $G_m=\mathbb{Z}/{m\mathbb{Z}}$ be the cyclic group of order $m$.
				Suppose that   $\sigma\in G_m$ acts  on $\mathbb{P}^1$
				by $\sigma (t)=\xi t$ (where $\xi$ is any $m$-th primitive unit root), and $\sigma\in G_m$ acts  on some elliptic curve $E_0$ by translation $\sigma(y)=y+p$, where $p$ is a torsion point of order $m$. Let $P:=(E_0\times \mathbb{P}^1)/{G_m}$, where $G_m$ acts on $E_0\times \mathbb{P}^1$ by the diagonal action. Then we have the following commutative diagram
				$$\xymatrix{E_0\times \mathbb{P}^1 \ar[rr]^-{\Pi}\ar[d] && P:=(E_0\times \mathbb{P}^1)/{G_m} \ar[d]^h\\
					E_0\ar[rr]^{m:1} &&  E:=E_0/G_m}
				$$
				
				Then $h:P\rightarrow E$ is a $\mathbb{P}^1$-bundle over the elliptic curve $E$,
				which admits two sections $C_0,C_\infty$ with $C_0^2=C_\infty^2=0$.
				Denote by $\Gamma$ the general fiber of $h$.
				Note also that the Picard group of $P$ is simply:
				$$\Pic(P)=\mathbb{Z}[C_0]\oplus h^*\Pic(E).$$
				\begin{lemma}\label{lem-5-2}
					Let $h:P\rightarrow E$ be the $\mathbb{P}^1$-bundle and above.
					For any $p\in P\setminus\{C_0\cup C_{\infty}\}$, let $\tau:\,\wt{P} \to P$ be the blowing-up centered at $p$, and $\mathcal{E}$ be the exceptional curve.
					Then for any integers $k\geq 2$ and $x\geq 2$ with $k+x$ being even,
					there exists a smooth divisor $\wt{R}\in \left|\tau^*\big((2g+2)C_0+2\Gamma\big)-2k\mathcal{E}\right|$ if $m\geq k+2$,
					where $g=\frac{k^2}{2}+xk+\frac{x^2-4}{2}$.
				\end{lemma}
				\begin{proof}
					The proof is the same as that of \cite[Lemma\,5.3]{ll24}.
					Indeed, it is enough to prove that the linear system $\left|\tau^*\big((2g+2)C_0+2\Gamma\big)-2k\mathcal{E}\right|$
					is base-point-free.
					By Reider's method \cite{reider88}, it suffices to shows that $L\cdot \wt{D} \geq 2$
					for any irreducible curve $\wt{D}\subseteq \wt P$,
					where
					$$L=\tau^*\big((2g+2)C_0+2\Gamma-K_P\big)-(2k+1)\mathcal{E}.$$
					Let $\wt{D}\subseteq \wt P$ be any irreducible curve.
					If $\wt{D}=\mathcal{E}$, or $\wt{D}$ is the strict transform of any fiber of $h:\,P\to E$,
					then one checks easily that $L\cdot \wt{D}>2.$
					Otherwise, $\wt{D}\sim_{\mathrm num} \tau^*(aC_0+b\Gamma)-\beta\mathcal{E}$ with $a>0$,
					i.e., the restriction $h|_{\wt{D}}:\, \wt{D} \to E$ is surjective.
					Thus
					\begin{equation}\label{eqn-5-1}
						0\leq 2p_a(\wt{D})-2 =(K_{\wt P}+\wt{D})\cdot \wt{D},\qquad\text{namely,}\qquad
						2(a-1)b\geq \beta(\beta-1).
					\end{equation}
					By direct computation,
					$$L\cdot \wt{D} = 2\big(a+(g+2)b\big)-(2k+1)\beta=2a+(k+x)^2b-(2k+1)\beta.$$
					If $\beta\geq \frac{k+x}{2x-1}$, then
					$$\begin{aligned}
						L\cdot \wt{D} &\,= 2+2(a-1)+(k+x)^2b-(2k+1)\beta\\
						&\,\geq 2+2\sqrt{2(k+x)^2(a-1)b}-(2k+1)\beta\\
						&\,\geq 2+2(k+x)\sqrt{\beta(\beta-1)}-(2k+1)\beta\\
						&\,> 2+2(k+x)\big(\beta-\frac12\big)-(2k+1)\beta\\
						&\,=2+(2x-1)\beta-k-x\geq 2;
					\end{aligned}$$
					if $\beta\leq \frac{k+x}{2x-1}$, then
					$$\begin{aligned}
						L\cdot \wt{D} &\,= 2+2(a-1)+(k+x)^2b-(2k+1)\beta\\
						&\,\geq \left\{\begin{aligned}
							&2+2(k+x)^2-\frac{(2k+1)(k+x)}{2x-1}> 2, &\quad&\text{if~}b\geq 2;\\
							&2+\beta(\beta-1)+(k+x)^2-(2k+1)\beta> 2, && \text{if~}b=1,\\
							&2, && \text{if $b=0$ and $a=1$};\\
							&2a-(2k+1)\beta \geq 2m-(2k+1)>2, && \text{if $b=0$ and $a\geq 2$};
						\end{aligned}\right.
					\end{aligned}$$
					Here we explain a little more about the case when $b=0$:
					if $a=1$, then $\wt{D}$ must be the strict transform of $C_0$ or $C_\infty$ by \cite[Lemma\,5.2]{ll24},
					which implies that $\beta=0$ since $p\in P\setminus\{C_0\cup C_{\infty}\}$,
					and hence $L\cdot \wt{D}=2$;
					if $a\geq 2$, then $\beta\leq 1$ by \eqref{eqn-5-1}, and $a\geq m$ by \cite[Lemma\,5.2]{ll24}.
				\end{proof}
				
				Let $\wt{R}\in \left|\tau^*\big((2g+2)C_0+2\Gamma\big)-2k\mathcal{E}\right|$ be any smooth divisor on $\wt P$ as in \autoref{lem-5-2}.
				It is clear that $\wt{R}$ is two-divisible, and hence one can construct a double cover
				$$\pi:\, S \to \wt{P},$$
				such that $\pi$ branches exactly over $\wt{R}$.
				The fibration $\tilde{h}:\,\wt{P} \to E$ induces a hyperelliptic fibration $f:\,S \to E$, whose general fiber is of genus
				$$g=\frac{k^2}{2}+xk+\frac{x^2-4}{2}.$$
				Since $\wt{R}$ is smooth, one checks easily that $S$ is minimal, and
				$$\begin{aligned}
					K_S^2&\,=2\Big(K_{\widetilde{P}}+\frac{\wt{R}}{2}\Big)^2=4(g-1)-2(k-1)^2=4(x+1)k+2x^2-14,\\
					\chi(\mathcal{O}_S) &\,=2\chi(\mathcal{O}_{\wt{P}})+\frac12\Big(K_{\widetilde{P}}+\frac{\wt{R}}{2}\Big)\cdot\frac{\wt{R}}{2}=g-\frac12k(k-1)=\Big(x+\frac12\Big)k+\frac{x^2-4}{2}.
				\end{aligned}$$
				Moreover, it is easy to check that $\wt{R}$ is ample.
				Hence $q(S)=q(\wt{P})=g(E)=1$, which implies that $f$ is the Albanese fibration of $S$.
				
				
				\begin{example}\label{exam-4-2}
					The Albanese fibrations in \autoref{exam-2} are all hyperelliptic. One can also construct non-hyperelliptic Albanese fibrations as follows.
				\end{example}
				
				By construction, there is anther fibration on the surface $P=(E_0\times \bbp^1)/G_m$:
				$$h_2:\, P \lra \bbp^1 \cong \bbp^1/G_m.$$
				The general fiber of $h_2$ is an elliptic curve isomorphic to $E_0$.
				It induces also fibrations on $\wt{P}$ and $S$ with the following diagram:
				$$\xymatrix{ S \ar[drr]^-{\pi}\ar@/_10mm/"3,5"_-{f} \ar@/^9mm/"2,7"^-{f_2}\\
					&& \wt{P} \ar[rr]_-{\tau} \ar[drr]_{\tilde h} \ar@/^5mm/"2,7"^-{\tilde h_2} && P \ar[d]^-{h} \ar[rr]_-{h_2} && \bbp^1\\
					&&&& E
				}$$
				Since the branch divisor of the double cover $\pi:\,S \to \wt{P}$ is
				$$\wt{R} \equiv \tau^*\big((2g+2)C_0+2\Gamma\big)-2k\mathcal{E},$$
				it follows that $f_2$ is a fibration of genus $g_2=m+1$.
				Note that the fibration $h_2:\,P \to \bbp^1$ is isotrivial with two multiple fibers $\Gamma_0=mC_0$ and $\Gamma_{\infty}=mC_{\infty}$ of multiplicity both equal to $m$.
				Let $\phi:\,\bbp^1 \to \bbp^1$ be a base change of degree two branched over
				$p_0=h(\Gamma_0)$ and $p$,
				where $p\in \bbp^1$ is a general point such that the fiber $f_2^{-1}(p)$ is smooth.
				Let $\ol{S}$ (resp. $\ol{P}$) be the normalization of $S\times_{\bbp^1} \bbp^1$ (resp. $\wt{P} \times_{\bbp^1}\bbp^1$).
				Then we get a commutative diagram as follows.
				$$\xymatrix{\ol{S} \ar[rr]_-{\bar\pi} \ar[d]^-{\Phi} \ar@/_5mm/"3,1"_-{\bar f}
					\ar@/^5mm/"1,5"^-{\bar f_2}
					&& \overline{P} \ar[rr]_-{\bar h_2} \ar[d]^-{\varphi}	&& \bbp^1 \ar[d]^-{\phi} \\
					S \ar[rr]^-{\pi} \ar[d]^-{f} && \wt{P} \ar[rr]^-{\tilde h_2} \ar[d]^-{\tilde h} && \bbp^1 \\
					E \ar@{=}[rr] && E }$$
				When $m$ ($\geq k+2$ required by \autoref{lem-5-2}) is even, the double cover $\Phi:\,\ol{S} \to S$ is branched over exactly over $f_2^{-1}(p)$,
				since the fiber $f_2^{-1}(p_0)=\pi^{-1}(\Gamma_0)$ is of multiplicity $m$.
				Let $\bar g$ be genus of a general fiber of $\bar f$.
				It follows immediately that
				$$\begin{aligned}
					&2\bar g-2=2(2g-2)+2m,\quad\Longrightarrow\quad \bar g=2g-1+m=k^2+2xk+x^2-5+m,\\
					&K_{\ol{S}}^2=2\Big(K_S+\frac{f_2^{-1}(p)}{2}\Big)^2=2(K_S^2+2m)=8(x+1)k+4x^2+4m-28,\\
					&\chi(\mathcal{O}_{\ol{S}})=2\chi(\mathcal{O}_S)+\frac12\Big(K_S+\frac{f_2^{-1}(p)}{2}\Big)\cdot \frac{f_2^{-1}(p)}{2} =(2x+1)k+x^2+\frac{m}{2}-4.
				\end{aligned}$$
				Note that $\bar\pi:\, \ol{S} \to \ol{P}$ is also a double cover branched over $\varphi^{-1}(\wt{R})$.
				As $\wt{R}$ is ample on $\wt{P}$, the divisor $\varphi^{-1}(\wt{R})$ is also ample on $\ol{P}$.
				Note also that $\bar h_2:\,\ol{P} \to \bbp^1$ is an elliptic fibration over $\bbp^1$. Hence $q(\ol{S})=q(\ol{P})=q(\wt{P})=1$.
				This implies that $\bar f:\, \ol{S} \to E$ is the Albanese fibration of $\ol{S}$.
				Moreover, the general fiber $\ol{F}$ of $\bar f$ is a double cover of $F$ branched over $m$ points, where $F=\Phi(\ol{F})$ is a general fiber of $f$ which is of genus $g$.
				Hence $\ol{F}$ is non-hyperelliptic by the Castelnuovo–Severi inequality (cf. \cite[Exercise\,V.1.9]{har77});
				indeed, if $\ol{F}$ were hyperelliptic, then there would be a double cover $\ol{F} \to F$
				as well as a double cover $\ol{F} \to \bbp^1$, which gives a contradiction to the Castelnuovo–Severi inequality.
				
				\begin{remark}
					Let $\lambda>4$.
					Fixing an integer $x\geq 2$ with $\frac{8(2x+3)}{4x+3}<\lambda$,
					let $k\geq 2$ be an integer with $k+x$ being even,
					and let $m=k+2$ (resp. $m=k+3$) if $k$ is even (resp. odd).
					Similarly as what we do in the proof of \autoref{thm-main} at the beginning of this section,
					one can obtain by \autoref{exam-4-2} a sequence of non-hyperelliptic Albanese fibrations $\bar f_k:\, \ol{S}_k \to E_k$,
					such that the genus $\bar g_k$ of the fibration $\bar f_k$ increases quadratically with $\chi(\mathcal{O}_{\ol S_k})$ and that $$\lim\limits_{k\to \infty} \frac{K_{\ol S_k}^2}{\chi(\mathcal{O}_{\ol S_k})}=\frac{8(2x+3)}{4x+3}<\lambda.$$
				\end{remark}

				

				\section{The Albanese fibrations reaching the upper bound}\label{sec-reaching}
				In this section, we would like to study the Albanese fibrations reaching the upper bound \eqref{eqn-main}.
				In particular, we will prove \autoref{thm-main-2} and \autoref{thm-main-3}.
				
				\begin{proof}[{Proof of \autoref{thm-main-2}}]
					Most of the conclusions have been in fact already proved in \autoref{sec-proof-main}.
					To be self-contained, we will also recall the proofs here again.
					
					First, from \autoref{lem-reduction} it follows that $q(S)=g(C)=1$ (and hence the base $C=E$ is an elliptic curve) and that the general fiber of $f$ is non-hyperelliptic.
					Since $g=3\chi(\mathcal{O}_S)+1\geq 16$ by assumption, it follows that
					$$K_S^2 =K_f^2\geq 4 \chi_f=4\chi(\mathcal{O}_S).$$
					Combining it with the assumption $K_S^2\leq 4\chi(\mathcal{O}_S)$,
					one proves that $K_S^2=4\chi(\mathcal{O}_S)$.
					In other words, the slope of $f$ is $\lambda_f:=\frac{K_f^2}{\chi_f}=4$.
					Thus the conclusions\,(iv) and (vi) follows from \cite[Theorem\,1.5]{lu-zuo-18}.
					It remains to prove that (v),
					i.e., the bielliptic surface $Y$ is the surface of type 6 appearing in Table\,2.1 in \autoref{sec-bielliptic}.
					This has been already proved in the proof of \autoref{prop-3-1} in \autoref{sec-proof-bielliptic}; see also \autoref{rem-bielliptic}.
					This completes the proof.
				\end{proof}
				
				\begin{proof}[{Proof of \autoref{thm-main-3}}]
					Let $S$ be an irregular minimal surface with $\chi(\mathcal{O}_S)=\chi\geq 5$ reaching the equality in \eqref{eqn-main}.
					By \autoref{thm-main-2}, the canonical model $S_{can}$ is a flat double cover of a bielliptic surface $Y$, and the branch divisor $R$ of $\pi:\, S_{can}\rightarrow Y$ admits at most negligible singularities.
					Moreover, the bielliptic surface $Y\cong (A\times B)/G$ is of type 6 in Table 2.1 in \autoref{sec-bielliptic},
					where $A$ can be an arbitrary elliptic curve, but $B$ is unique up to isomorphism. Hence there is exactly one parameter for $Y$.
					
					Now we consider the parameters for the branch divisor $R$. Since $\gamma=|G|=9$, by \eqref{eqn-3-10} and \eqref{eqn-3-11}, we have
					\begin{equation}\label{}
						b=\frac{2}{3}, \quad  a=\frac{2\chi(\mathcal{O}_S)}{3},
						\quad R\sim \frac{2\chi(\mathcal{O}_S)}{3}A+\frac{2}{3}B.
					\end{equation}
					Since $R$ is ample and $K_Y\sim 0$, we see $h^1(\mathcal{O}_Y(R))=h^2(\mathcal{O}_Y(R))=0$. By the Riemann-Roch Theorem, we get
					
					$$h^0(\mathcal{O}_Y(R))=\chi(\mathcal{O}_Y)+\frac{R^2}{2}=4\chi(\mathcal{O}_S).$$
					On the other hand,    a general member in $|\mathcal{O}_Y(R)|$ is smooth by \autoref{exam-1}.
					Since  the parameter spaces for $Y$ and $R$ are both irreducible,
					$\mathcal{M}$ is also irreducible. The  dimension of $\mathcal{M}$ is equal to
					\[1+\dim |\mathcal{O}_Y(R)|=4\chi(\mathcal{O}_S)=4\chi.\qedhere\]
				\end{proof}
				
				Most of the conclusions in \autoref{thm-main-2} and \autoref{thm-main-3}
				can be extended to the minimal irregular surfaces whose Albanese fibrations are bielliptic.
				
				\begin{theorem}\label{thm-moduli}
					Let $S$ be a minimal surface of general type with $q(S)=1$ and $K_S^2 = 4\chi(\mathcal{O}_S)$. Suppose that its Albanese fibration $f:\,S \to E$  is a bielliptic fibration of genus $g\geq 5$.
					\begin{enumerate}[$(1)$]
						\item The canonical model $S_{can}$ is a flat double cover of a bielliptic surface $Y$  whose branch divisor $R$  admits at most negligible singularities.
						\item The genus $g$ is linearly bounded from above as follows:
						\begin{equation}\label{eqn-5-12}
							g\leq \left\{\begin{aligned}
								&\chi(\mathcal{O}_S)+1, &\quad&\text{if~} Y ~\text{is of type $1$, $3$, $5$, or $7$ in Table\,2.1 in \autoref{sec-bielliptic}};\\
								&2\chi(\mathcal{O}_S)+1, &\quad&\text{if~} Y ~\text{is of type $2$ or $4$ in Table\,2.1 in \autoref{sec-bielliptic}};\\
								&3\chi(\mathcal{O}_S)+1 &\quad&\text{if~} Y ~\text{is of type $6$ in Table\,2.1 in \autoref{sec-bielliptic}}.
							\end{aligned}\right.
						\end{equation}
						
						\item For any fixed integer $\chi\geq 2$, let $\mathcal{M}_{\chi}^i$ be the moduli space of minimal irregular surfaces $S$ of general type with $q(S)=1, \chi(\mathcal{O}_S)=\chi, K_S^2=4\chi(\mathcal{O}_S)$ whose Albanese map is a bielliptic fibration of genus $g\geq 5$ reaching the equality in \eqref{eqn-5-12} with $Y$ being of type $i$ in Table\,2.1 in \autoref{sec-bielliptic}.
						Then $\mathcal{M}_{\chi}^i$ is irreducible and
						\begin{equation}\label{eqn-5-13}
							\dim \mathcal{M}_{\chi}^i= \left\{\begin{aligned}
								&4\chi+1, &\quad&\text{if~} 1\leq i\leq 2;\\
								&4\chi, &\quad&\text{if~} 3\leq i\leq 7.
							\end{aligned}\right.
						\end{equation}
					\end{enumerate}
				\end{theorem}
				\begin{proof}
					(1) This is in fact contained in the proof of \autoref{prop-3-1}.
					Indeed, since $g\geq 5$ and $E$ is an elliptic curve,
					it follows that the slope of $f$ is $\lambda_f=\frac{K_S^2}{\chi(\mathcal{O}_S)}=4$.
					Hence our conclusion follows from \cite[Theorem\,2.1]{barja01}.

					(2) The proof is similar to that of \autoref{prop-3-1}.
					We use the same notations as in the proof of \autoref{prop-3-1}.
					By \eqref{eqn-3-11} and \eqref{eqn-modui-need}, we have
					$$g=\frac{2\chi(\mathcal{O}_S)}{b}+1,\qquad \text{and}\qquad b\geq \frac{2}{\beta}.$$
					Thus the upper bound follows from the description of $\beta$ in each type bielliptic surfaces in Table\,2.1 in \autoref{sec-bielliptic}.
					
					(3) First we show that $\mathcal{M}_{\chi}^i$ is nonempty: simialr to the construction in \autoref{exam-1},  take $Y$ to be of type $i$,
					$a=\frac{4\chi}{b\gamma}$, where $\gamma=|G|$ is order of the group $G$ defining the bielliptic surface as in \autoref{sec-bielliptic}, and
					$$b=\left\{\begin{aligned}
						&2, &\quad& \text{if $i=1, 3, 5$ or $7$};\\
						&1, && \text{if $i=2$ or $4$};\\
						&\frac{2}{3}, && \text{if $i=6$}.
					\end{aligned}\right.$$
					Then one can construct a minimal surface $S$ of general type with
					$q(S)=1, \chi(\mathcal{O}_S)=\chi, K_S^2=4\chi$ whose Albanese map is a bielliptic fibration of genus $g\geq 5$ and reaching the equality in \eqref{eqn-5-12}.
					
					To count the dimension of $\mathcal{M}_{\chi}^i$, we can also mimic what we do in the proof of \autoref{thm-main-3}.
					For the bielliptic surface $Y$, there are 2 parameters  if  $1\leq i\leq 2$
					and $1$ parameter if $3\leq i\leq 7$. As calculated  in the proof of \autoref{thm-main-3}, we see that  $\dim |\mathcal{O}_Y(R)|=4\chi -1$ (which is independent of the type of $Y$)
					and that a general a general member in $|\mathcal{O}_Y(R)|$ is smooth.
					Hence $\mathcal{M}_{\chi}^i$ is irreducible of dimension $4\chi+1$ if $1\leq i\leq 2$ and of dimension $4\chi$ if $3\leq i\leq 7$.
				\end{proof}
				
				\begin{remark}
					Let $S$ be a minimal  surface of general type with $q(S)=1$, $K_S^2 = 4\chi(\mathcal{O}_S)$, and $f:\,S \to E$ be its Albanese fibration whose general fiber is non-hyperelliptic of genus $g\geq 5$.
					
					(1) By \autoref{lemma-rkU-1},
					if either $g\geq 16$ or $\rank \mathcal{U}=1$, where $\mathcal{U}$ is the unitary summand in the Fujita decomposition in \eqref{eqn-fujita},
					then
					$f$ is a bielliptic fibraiton and satisfy the condition of \autoref{thm-moduli}.

					(2) Suppose that $f$ is bielliptic of genus $g\geq 5$.
					Then it always holds that $\rank\mathcal{U}=1$, where $\mathcal{U}$ is the unitary summand  in the Fujita decomposition as above.
					
					\begin{proof}
						Consider the following commutative diagram in  \autoref{lemma-rkU-1}
						$$\xymatrix{S \ar[rr]^-{\pi} \ar[d]_-{f} && Y \ar[d]^-{h}\\
							E \ar[rr]^-{=} && E}$$
						Let $\mathcal{L}$ be a line bundle on $E$ such that  $\mathcal{L}^{\otimes 2}\cong \mathcal{O}_Y(R)$ , then we have
						$$f_*\omega_f=h_*(\pi_*\omega_f)=h_*(\omega_Y\oplus \mathcal{L})=h_*(\omega_Y)\oplus h_*{\mathcal{L}}$$
						Since $\deg h_*(\omega_Y)=0$,  we have $\rank  \mathcal{U}\geq 1$.
						On the other hand, we have showed in the proof of \autoref{lemma-rkU-1} that $\rank  \mathcal{U}\leq 1$.
						Hence $\rank\mathcal{U}=1$ as required.
					\end{proof}
				\end{remark}	
	
\vspace{1cm}
{\bf\noindent Data availability~} The author declares that the manuscript has no associated data.
\vspace{8mm}

{\bf\large \noindent Declarations}
\vspace{8mm}

\noindent On behalf of all authors, the corresponding author states that there is no conflict of interest.

					%


			\end{document}